\documentclass[12pt]{article}
\usepackage{amsmath, amsthm, amsfonts, enumerate}

\def\fullpage {
\addtolength{\topmargin}{-2 cm}
\addtolength{\oddsidemargin}{-0.9cm} \addtolength{\textwidth}{+2 cm}
\addtolength{\textheight}{+4 cm}}
\fullpage
\newtheorem{thm}{Theorem}
\newtheorem*{thmall}{Asymmetric Local Lemma}
\newtheorem{lemma}[thm]{Lemma}
\newtheorem{cor}[thm]{Corollary}

\newtheorem{claim}[thm]{Claim}
\theoremstyle{remark}

\theoremstyle{definition}
\newtheorem{definition}[thm]{Definition}
\newcommand{\Ex}{\mathop{\bf E\/}}
\newcommand{\I}{\mathop{\bf I\/}}
\newcommand{\Var}{\mathop{\bf Var\/}}
\newcommand{\comp}[2]{\substack{{\text{ \tiny #2 }} \\ #1}}
\numberwithin{equation}{section}

\title{List Coloring Triangle-Free Hypergraphs}

\begin{document}

\author{
\quad{Jeff Cooper}
\thanks{Department of Mathematics, Statistics, and Computer
Science, University of Illinois at Chicago, Chicago, IL 60607, USA;  email:
jcoope8@uic.edu}
\quad{Dhruv Mubayi}
\thanks{Department of Mathematics, Statistics, and Computer Science, University of Illinois at Chicago, Chicago IL 60607, USA; research supported in part by NSF grant 0969092; email:  mubayi@uic.edu}
}

\maketitle

\begin{abstract}
A triangle in a hypergraph is a collection of  distinct vertices $u,v,w$ and  distinct edges $e,f,g$ with $u,v \in e$, $v,w \in f$,  $w,u \in g$ and $\{u,v,w\} \cap e \cap f \cap g=\emptyset$. Johansson~\cite{johansson} proved that every triangle-free graph with maximum degree $\Delta_2$ has list chromatic number $O(\Delta_2/\log\Delta_2)$. Frieze and the second author~\cite{fmcoloring3} proved that every  {\it linear} (meaning that every two edges share at most one vertex)  triangle-free  triple system with maximum degree $\Delta_3$ has chromatic number $O(\sqrt{\Delta_3/\log \Delta_3})$.  The restriction to linear triple systems was crucial to their proof. 

We provide a generalization of these results. The $i$-degree of a vertex in a hypergraph is the number of edges of size $i$ containing it. 
We prove that every triangle-free hypergraph of rank three (edges have size two or three) with maximum 3-degree $\Delta_3$ and maximum 2-degree $\Delta_2$ has list chromatic number at most 
\[ c \,\max\left\{  \frac{\Delta_2}{\log{\Delta_2}}\, , \left(\frac{\Delta_3}{\log{\Delta_3}}\right)^{\frac{1}{2}} \right\},
\] 
for some absolute positive constant $c$.

Thus our result removes the {\it linear} restriction from~\cite{fmcoloring3} and applies to the broader class of rank three hypergraphs, while reducing to the (best possible) result~\cite{johansson} for graphs.  As an application, we prove that if ${\cal C}_3$ is the collection of 3-uniform triangles, then the Ramsey number $R({\cal C}_3, K_t^3)$ satisfies
\[
\frac{a t^{3/2}}{(\log t)^{3/4}} \leq R({\cal C}_3, K_t^3) \leq  \frac{bt ^{3/2}}{(\log t)^{1/2}}
\]
for some positive constants $a$ and $b$. The upper bound makes progress towards the recent conjecture of Kostochka, the second author, and Verstra\"ete~\cite{kmvr3} that $R(C_3, K_t^3) = o(t^{3/2})$ where $C_3$ is the linear triangle. 
\end{abstract}

\section{Introduction}
A hypergraph $H = (V, E)$ is a tuple consisting of a set of vertices $V$
and a set of edges $E$, which are subsets of $V$.
The hypergraph has rank $k$ if every edge contains at most $k$ vertices
and is called $k$-uniform if every edge contains exactly $k$ vertices.
A proper coloring of $H$ is an assignment of colors
to the vertices so that no edge is monochromatic.
The chromatic number of $H$, $\chi(H)$, is the minimum number of colors
needed in a proper coloring of $H$.

The chromatic number of graphs ($2$-uniform hypergraphs) has been studied extensively.
A greedy coloring algorithm can be used to show that for any graph $G$ with maximum 
degree $\Delta$, $\chi(G) \leq \Delta + 1$; this bound is tight for complete graphs and 
odd cycles. Brooks \cite{brooks} extended this by showing that if $G$ is not a complete graph 
or an odd cycle, then $\chi(G) \leq \Delta$. 

A natural question to ask is what other structural properties
can be put on a graph to decrease its chromatic number.
One approach is to fix a graph $K$ and consider the family of graphs which
contain no copy of $K$. 
For example, if $K$ is a tree on $e$ edges and $G$ contains no copy of $K$,
then $\chi(G) \leq e$; this follows from the fact that if $G$ contains no copy of $K$,
then $G$ contains a vertex of degree at most $e-1$ (see \cite{west}, pg. 70).

When $K$ is a cycle, the problem becomes more difficult. 
Kim \cite{kimc4} showed that if $G$ contains no $4$-cycles or $3$-cycles, then 
$\chi(G) \leq (1+o(1))\Delta/\log\Delta$ as $\Delta \to \infty$,
which is within a factor of $2$ of the best possible bound.
Shortly after, Johansson \cite{johansson} showed that if $G$ contains no $3$-cycles,
then $\chi(G) \leq O(\Delta/\log\Delta)$.
Using Johansson's result, Alon, Krivelevich, and Sudakov \cite{aloncoloringsparse} showed that if
$K$ is any graph containing a vertex $x$ such that $K-x$ is bipartite, then
$\chi(G) \leq O(\Delta/\log\Delta)$.

Some analogous results for hypergraphs are known. 
Using the local lemma, one can show that $\chi(H) \leq O(\Delta^{1/(k-1)})$
for any $k$-uniform hypergraph $H$.
Bohman, Frieze, and the second author \cite{bfmhfree} showed that if $K$ is a 
fixed $k$-uniform hypertree on $e$ edges and 
$H$ is a $k$-uniform hypergraph containing no copy of $K$, then $\chi(H) \leq 2(k-1)(e-1) + 1$;
Loh \cite{lohtfree} improved this to $\chi(H) \leq e$, matching the result for graphs.

A hypergraph is \emph{linear} (or contains no $2$-cycles) if any two of its edges 
intersect in at most one vertex. A \emph{triangle} in a linear hypergraph is a set of three
pairwise intersecting edges with no common point.
In \cite{fmcoloring3}, Frieze and the second author showed that if $H$ is a $3$-uniform, linear,
triangle-free hypergraph, then $\chi(H) \leq O(\sqrt{\Delta}/\sqrt{\log \Delta})$.
They subsequently removed the triangle-free condition and generalized their 
result from $3$ to $k$,
showing that $\chi(H) \leq O( (\Delta/\log\Delta)^{1/(k-1)})$ for any $k$-uniform, linear
hypergraph $H$. As shown in \cite{bfmhfree}, these results are tight apart from the implied
constants.

\subsection{Our Result}
Our contribution is to remove the linear condition from~\cite{fmcoloring3}.
However, in doing so, we also widen the definition of a triangle.

\begin{definition}
A \emph{triangle} in a hypergraph $H$ is a set of three distinct edges $e,f,g \in H$
and three distinct vertices $u,v,w \in V(H)$ such that $u,v \in e$, $v,w \in f$, $w, u \in g$
and $\{u, v, w\} \cap e \cap f \cap g$.
\end{definition}
\noindent For example, the three triangles in a $3$-uniform hypergraph are the loose triangle
$C_3=\{abc, cde, efa\}$, $F_5=\{abc, bcd, aed\}$, and $K_4^-=\{abc, bcd, abd\}$.

Given a set $L(v)$ of colors for every vertex $v \in V(H)$,
a proper list coloring of $H$ is a proper coloring where
every vertex $v$ receives a color from $L(v)$.
The list chromatic number of $H$, $\chi_l(H)$, is the minimum $l$
so that if $|L(v)| \geq l$ for all $v$, then $H$ has a proper list coloring.
It is not hard to see that $\chi(H) \leq \chi_l(H)$.
As in \cite{kimc4} and \cite{johansson},
our main theorem can be stated in terms of list chromatic number.
If $H$ is a rank $k$ hypergraph and $i \leq k$, 
the $i$-degree of a vertex $v$ is the number of size $i$ edges containing $v$.
\begin{thm}\label{mainthm}
Suppose $H$ is a rank $3$, triangle-free hypergraph with 
maximum $3$-degree $\Delta$ and maximum $2$-degree $\Delta_2$.
Then
\[
 \chi_l(H) \leq c_1 \max\{ (\frac{\Delta}{\log{\Delta}})^{\frac{1}{2}}, \frac{\Delta_2}{\log{\Delta_2}} \},
\] 
for some constant $c_1$.
\end{thm}

Theorem \ref{mainthm} generalizes the results of \cite{johansson} and \cite{fmcoloring3}.  Additionally, it  strengthens~\cite{fmcoloring3} by removing the linear hypothesis, which was a crucial ingredient in the proof. 
As mentioned above, for $n$-vertex 3-uniform hypergraphs $H$ with maximum degree $\Delta$, one can easily show that the independence number of $H$ is $\Omega(n/\sqrt \Delta)$ and  $\chi(H)= O(\sqrt{\Delta})$; however, adding a local restriction to the hypergraph in order to significantly improve either of these bounds appears to be a hard problem.  There are two conjectures in this regard.  De Caen \cite{decaen} conjectured that if we add the hypothesis that every vertex subset $S$ spans at most $c|S|^2$ edges (for some fixed constant $c$), and $\Delta=\Theta(n)$, 
then the lower bound on the independence number can be improved by a factor that tends to infinity with $\Delta$.
More recently, \cite{fmcoloring3} conjectured that if there is a fixed hypergraph $F$ with $F \not\subset H$, then $\chi(H)< c_F \sqrt{\Delta/\log \Delta}$.  
Guruswami and Sinop~\cite{guruswamisinop} showed that this conjecture implies certain hardness results in computer science.

We prove Theorem \ref{mainthm} by using a semi-random algorithm to properly color the hypergraph.
Our algorithm is similar to the algorithm in~\cite{fmcoloring3}, however,
several new ideas  are developed to deal with the non-linear case.
At each iteration, we randomly color a few of the vertices.
When a vertex in a $3$-edge is colored $c$, we add a $c$-colored $2$-edge between the remaining two
vertices to record the fact that those two vertices cannot both be colored $c$
in the future. \cite{fmcoloring3} assumed the hypergraph was linear, which
implied that at most one such $2$-edge could be added between two vertices.
Here we maintain a $2$-graph for every color and
allow two vertices to share an edge in multiple graphs. This allows us to
extend our algorithm to rank $3$ hypergraphs: for each $2$-edge in the original
hypergraph, we simply add a copy of that $2$-edge to every color graph.
After several iterations, we color the remaining vertices with the asymmetric version
of the local lemma. This prevents the $3$-edges from becoming monochromatic, while
also enforcing the constraints from the $2$-graphs.

\subsection{Application to Hypergraph Ramsey Numbers}

Let ${\cal C}^r_3$ be the collection of $r$-uniform hypergraph triangles. 
Notice that for graphs, $C^2_3$ consists of only the $3$-vertex cycle, and
for triple systems, $\mathcal{C}^3_3 = \{C_3, F_5, K_4^-\}$.
The hypergraph Ramsey number $R({\cal C}^r_3,K_t^r)$ is the smallest $n$
so that in every red-blue coloring of the edges of the complete $r$-uniform
hypergraph $K_n^r$, there exists a red triangle or a blue $K_t^r$.  
Ajtai-Koml\'os-Szemer\'edi~\cite{aksr3t} and Kim~\cite{kimr3t} proved that $R(\mathcal{C}^2_3, K_t^2)=\Theta(t^2/\log t)$.

In \cite{kmvr3}, Kostochka, the second author, and Verstra\"ete proved a
version of this result for $r=3$.
In this setting, $R(C_3, K_t^3)$ is the smallest $n$ so that in every red-blue coloring of
the edges of the complete $3$-uniform hypergraph $K_n^3$, there exists a
red $C_3$ or a blue $K_t^3$.
\cite{kmvr3} showed that there exist constants $a, b$
such that 
\[
\frac{a t^{3/2}}{(\log t)^{3/4}} \leq R(C_3, K_t^3) \leq b t ^{3/2},
\]
and they conjectured that the upper bound could be reduced to $o(t^{3/2})$.
We prove a weaker form of this conjecture, namely that
$R({\cal C}^3_3,  K_t^3) = O(t^{3/2}/\sqrt{\log t})$.
Since the $C_3$-free construction given in \cite{kmvr3} is also $F_5$
and $K_4^-$ free, this implies that for some constants $a$ and $b$,
\[
\frac{a t^{3/2}}{(\log t)^{3/4}} \leq R({\cal C}^3_3, K_t^3) \leq b \frac{t ^{3/2}}{(\log t)^{1/2}}.
\]
\subsection{Organization}
In Section 2, we present the probabilistic tools we will need to analyze our algorithm.
In Section \ref{sec_algorithm}, we describe our algorithm. 
The presentation is similar to Vu's description in \cite{vucoloring} of Johansson's algorithm. 
Section \ref{sec_analysis} contains an analysis of our algorithm. 
This analysis does not use triangle-free anywhere, but is instead based on
parameters which can be given to the algorithm.
In Section \ref{sec_trianglefree}, we show how triangle-free can be used to set
these parameters in a way that implies Theorem \ref{mainthm}.

\section{Tools}
\subsection{Local Lemma}
\begin{thmall}[\cite{molloyreed}]\label{asyll}
Consider a set $\mathcal{E} = \{A_1, \dots, A_n\}$ of (typically bad) events that such each $A_i$ is mutually
independent of $\mathcal{E}-(\mathcal{D}_i \cup A_i)$, for some $\mathcal{D}_i \subset \mathcal{E}$. If for each $1 \leq i \leq n$
\begin{itemize}
	\item $\Pr[A_i] \leq 1/4$, and
	\item $\sum_{A_j \in \mathcal{D}_i} \Pr[A_j] \leq 1/4$,
\end{itemize}
then with positive probability, none of the events in $\mathcal{E}$ occur.
\end{thmall}

\subsection{Concentration Theorems}
The first result is due to Hoeffding \cite{hoeffding}.
\begin{thm}\label{hoeffding}
Suppose that $X = X_1 + \dots + X_m$, where the $X_i$ are independent random variables
satisfying $|X_i| \leq a_i$ for all $i$. Then for any $t > 0$,
\[
\Pr[X \geq \Ex[X] + t] \leq e^{-\frac{2 t^2}{\sum_{i=1}^m a_i^2}},
\]
and
\[
\Pr[X \leq \Ex[X] - t] \leq e^{-\frac{2 t^2}{\sum_{i=1}^m a_i^2}}.
\]
\end{thm}
\noindent
We will also use the following theorem, which is Theorem 2.7 from \cite{concentration}.
\begin{thm}\label{hoeffdingvar}
Suppose that $X = X_1 + \dots + X_m$, where the $X_i$ are independent random variables
satisfying $X_i \leq \Ex[X_i] + b$ for all $i$. Then for any $t > 0$,
\[
\Pr[X \geq \Ex[X]+t] \leq e^{-\frac{t^2}{2\Var[X]+bt}}.
\]
\end{thm}
\noindent
McDiarmid \cite{mcdiarmid} proved the following generalization of Theorem \ref{hoeffding}.
\begin{thm}\label{azuma}
Let $Z_1,\dots,Z_n$ be independent random variables, with $Z_i$ taking
values in a set $\mathcal{A}_i$ for each $i$. Suppose that the (measurable) function
$g: \prod \mathcal{A}_k \to \mathbb{R}$ satisfies $|g(x)-g(x')| \leq d_i$
whenever the vectors $x$ and $x'$ differ only in the $i^{th}$ coordinate.
Let $W$ be the random variable $g(Z_1,\dots,Z_n)$. Then for any $t > 0$,
\[
\Pr[W > \Ex[W] + t) \leq e^{-2t^2 \sum_{i=1}^n d_i^2}
\]
\end{thm}

\noindent 
Note that in the above theorem, we may view $\prod\mathcal{A}_k$ as a
probability space induced by the random variables $Z_1,\dots,Z_n$.
We will use the following corollary, which resembles Theorem 7.2 
from \cite{concmeasure}.
\begin{cor}\label{bdbad}
Let $X_1,\dots,X_n$ be independent random variables, with $X_i$ taking
values in a set $\mathcal{B}_i$ for each $i$.
Let $\mathcal{A}_1, \dots, \mathcal{A}_n$ be events, 
where each $\mathcal{A}_i \subset \mathcal{B}_i$. 
Set $\mathcal{A} = \prod_{i=1}^n \mathcal{A}_i$.
Suppose that the (measurable) function
$f: \prod \mathcal{B}_k \to \mathbb{R}$ is non-negative and 
satisfies $|f(x)-f(x')| \leq d_i$
for any two vectors $x, x' \in \mathcal{A}$ differing only in the $i^{th}$ coordinate.
Let $Y$ be the random variable $f(X_1,\dots,X_n)$.
Then
\[
\Pr[Y > \Ex[Y]/\Pr[\mathcal{A}] + t] \leq e^{-2t^2/\sum_{i=1}^n d_i^2} + \Pr[\bar{\mathcal{A}}].
\]
\end{cor}
\begin{proof}
Define $g:\mathcal{A} \to \mathbb{R}$ by $g(x) := f(x)$
(in other words, $g = f|\mathcal{A}$).
For each $i$, let $Z_i:X_i^{-1}(\mathcal{A}_i)\to\mathcal{A}_i$ 
be the random variable with $Z_i(s) = X_i(s)$ for all $s \in X_i^{-1}(\mathcal{A}_i)$.
Let $W$ be the random variable $g(Z_1,\dots,Z_n)$.
Since the $X_i$ are independent, the $Z_i$ are also independent, so
we will be able to apply Theorem \ref{azuma} to bound $\Pr[W > \Ex[W] + t]$.

By total probability and the non-negativity of $f$,
\[
\Ex[Y] = \Ex[Y | \mathcal{A}]\Pr[\mathcal{A}] 
 + \Ex[Y | \bar{\mathcal{A}}]\Pr[\bar{\mathcal{A}}]
\geq \Ex[Y | \mathcal{A}]\Pr[\mathcal{A}] 
\]
so
\[
\Ex[W] = \Ex[Y | \mathcal{A}]  \leq \Ex[Y]/\Pr[\mathcal{A}].
\]
Combining this with Theorem \ref{azuma} implies
\begin{align*}
\Pr[Y > \frac{\Ex[Y]}{\Pr[\mathcal{A}]} + t]
&= \Pr[Y > \frac{\Ex[Y]}{\Pr[\mathcal{A}]} + t | \mathcal{A}]\Pr[\mathcal{A}] + \Pr[Y > \frac{\Ex[Y]}{\Pr[\mathcal{A}]} + t | \bar{\mathcal{A}}]\Pr[\bar{\mathcal{A}}] \\
&\leq \Pr[Y > \frac{\Ex[Y]}{\Pr[\mathcal{A}]} + t | \mathcal{A}] + \Pr[\bar{\mathcal{A}}] \\
&\leq \Pr[Y > \Ex[Y | \mathcal{A}] + t | \mathcal{A}] + \Pr[\bar{\mathcal{A}}] \\
&= \Pr[W > \Ex[W] + t] + \Pr[\bar{\mathcal{A}}] \\
& \leq e^{-2t^2/\sum_{i=1}^n d_i^2} + \Pr[\bar{\mathcal{A}}].
\end{align*}
\end{proof}

\section{Coloring Algorithm}\label{sec_algorithm}
The input to our algorithm is a rank $3$ hypergraph with maximum $3$-degree $\Delta$
and maximum $2$-degree $\Delta_2$.
Let $H$ denote the input hypergraph restricted to its size $3$ edges,
and let $G$ denote the input hypergraph restricted to its size $2$ edges.
At the beginning, each vertex $u$ has a list $C(u)$ of acceptable colors.
We assume $|C(u)| = C$ for all vertices $u$.
For each vertex $u$ and color $c$, we set 
\[
p^0_u(c) = 
\begin{cases}
1/C, &\text{if $c \in C(u)$}\\
0, &\text{if $c \notin C(u)$}.
\end{cases}
\]
We define a parameter $\hat{p}$, which will serve as an upper bound on the weights $p^i_u(c)$.
Set $W^0(u) = \{p^0_u(c): c \in \cup_v C(v)\}$. We start with the hypergraph
$H^0 = H$ and the collection $\{W^0(u)\}_u$.
For each color $c$, we also construct a graph $G^0_c$, which is
initially a copy of the $2$-graph $G$. 
Finally, we assign to each vertex an empty set $B^0(u)$.

At the $(i+1)^{th}$ step, $i = 0, 1, \dots, T-1$, our input to
the algorithm is a quadruple, $(H^i, \{G^i_c\}_c, \{W^i_u\}_u, \{B^i(u)\}_u)$.
We generate a small random set of colors at each vertex $u$ as follows:
For each color $c$, we choose $c$ with probability $\theta p_u^i(c)$.
Let 
\[
\gamma^i_u(c) = 
\begin{cases} 
1, &\text{if $c$ is chosen at $u$},\\
0, &\text{otherwise.}
\end{cases}
\]
Note that the $\gamma^i_u(c)$ are independent random variables.

Consider a vertex $u$. 
We define the set of colors lost at $u$ as
\[
L(u) = \{c : \exists e \in E(H^i) \cup E(G^i_c) 
 \text{ such that } u \in e \text{ and } \gamma^i_v(c)=1 \text{ } \forall v \in e-u \}.
\]
We say a color $c$ \emph{survives} at $u$
if $c \notin B^i(u)$ and $c \notin L(u)$.
For $c \notin B^i(u)$, we define
\begin{equation}
  q^i_u(c) :=
  \Pr[c \text{ survives at } u] = 
  \Pr[ \bigcap_{uvw \in H^i} (\gamma^i_v(c)=0 \cup \gamma^i_w(c)=0)
       \bigcap_{uv \in G^i_c} \gamma^i_v(c)=0 ].
\end{equation}
In other words, if $c \notin B^i(u)$, then $q^i_u(c) = \Pr[c \notin L(u)]$.
Note that at the $(i+1)^{th}$ step, $q^i_u(c)$ is a fixed number, which
can be computed given $H^i$, $G_c^i$, and all of the $p^i_v(c)$;
it does not depend on the random variables $\gamma^i_u(c)$.
In the analysis below, we will use the bound
\begin{align}\label{qbound}
q^i_u(c) 
&= 1- \Pr[ \bigcup_{uvw \in H^i} (\gamma^i_v(c)=1 \cap \gamma^i_w(c)=1)
  \bigcup_{uv \in G^i_c} \gamma_v(c)=1 ] \notag \\
&\geq 1-\sum_{uvw \in H^i} \theta^2 p^i_v(c)p^i_w(c)
 -\sum_{uv \in G^i_c} \theta p^i_v(c).
\end{align}
Let $\I[X]$ denote the $0,1$ indicator variable for the event $X$.
Define $p^{i+1}_u(c)$ as:
\begin{itemize}
\item If $p^i_u(c)/q^i_u(c) < \hat{p}$ and $c \notin B^i(u)$, then
 \begin{equation}\label{flip1}
   p^{i+1}_u(c) = p^i_u(c)\frac{\I[c \text{ survives at } u]}{\Pr[c \text{ survives at } u]}
   =
   \begin{cases} 
    p^i_u(c)/q^i_u(c) , &\text{if $c$ is survives at $u$},\\
    0, &\text{else}.
   \end{cases}
 \end{equation}
\item If $p^i_u(c)/q^i_u(c) \geq \hat{p}$ or $c \in B^i(u)$, 
  then we toss a biased coin with $\Pr[Head] = p^i_u(c)/\hat{p}$.
  We then set
  \[
  \eta^i_u(c) = \I[Head],
  \]
  and
  \begin{equation}\label{flip2}
    p^{i+1}_u(c) = p^i_u(c)\frac{\I[Head]}{\Pr[Head]}
    =
    \begin{cases}
      \hat{p}, &\text{if $\eta^i_u(c) = 1$} \\
      0, &\text{else}.
    \end{cases}
  \end{equation}
\end{itemize}
Crucially, \eqref{flip1} and \eqref{flip2} imply 
\begin{equation}\label{expuc}
\Ex[p^{i+1}_u(c)] = p^i_u(c).
\end{equation}

Color $u$ with $c$ if $c$ survives at $u$ and $\gamma^i_u(c)=1$ (if there are multiple such $c$,
pick one arbitrarily). Let $U^{i+1}$ denote the set of uncolored vertices
in $H$ after the iteration $i$.
Let $H^{i+1}$ be the hypergraph induced from $H$ by $U^{i+1}$, let $B^{i+1}(u) = \{c : p^{i+1}_u(c)=\hat{p}\}$,
and let $W^{i+1}_u = \{p^{i+1}_u(c)\}$. 
To form $G_c^{i+1}$, start with $G_c^i$, and for each triple
$u,v,w \in U^{i}$ with $u,v \in U^{i+1}$, $uv \notin G_c^{i}$, and $w$ colored $c$,
add an edge $uv$ to $G^{i+1}_c$. Then delete any vertex from $G_c^{i+1}$ that is not in $U^{i+1}$.

Observe that if $uvw$ is an edge in $H^i$ and $u$ and $v$ are both colored $c$ in the current
round, then $p^{i+1}_w(c) \in \{0, \hat{p}\}$; in particular, $c$ is never considered for $w$
in a future round. Similarly, if $vw \in G^i_c$ and $v$ is colored with $c$ in the current round,
then $c$ is never considered for $w$ in the future. Thus the algorithm always maintains
a proper partial coloring of $H$.

After $T$ iterations, some vertices will remain uncolored. We color these
in one final step, which is described in Section \ref{sec_finalstep}.

\subsection{Parameters and Notation}
We summarize all of the variables used in the algorithm and its analysis in the two tables below. 
The first table contains descriptions of the independent variables in our algorithm.
We set them for one family of hypergraphs in Section \ref{sec_trianglefree}, 
when we prove that our algorithm works for triangle-free hypergraphs. 
The values of the remaining parameters are defined in the second table.

Our algorithm requires that the parameter $\omega_0$ satisfy the following properties:
\begin{itemize}
\item For any edge $uvw$ in $H^i$ and any color $c$,
\begin{equation}\label{prcnotinuvw}
\Pr[c \notin L(u) \cup L(v) \cup L(w)] \leq q^i_u(c)q^i_v(c)q^i_w(c)(1+1/\omega_0).
\end{equation}
\item For any color $c$ and any pair $u,v$ with $uvw \in H^i$ for some $w$,
\begin{equation}\label{prcnotinuvh}
\Pr[c \notin L(u) \cup L(v)] \leq q^i_u(c)q^i_v(c)(1+1/\omega_0).
\end{equation}
\item For any color $c$ and any edge $uv$ in $G^i_c$,
\begin{equation}\label{prcnotinuv}
\Pr[c \notin L(u) \cup L(v)] \leq q^i_u(c)q^i_v(c)(1+1/\omega_0).
\end{equation}
\end{itemize}
The parameters $\omega_1$ through $\omega_6$ are error terms used
in the analysis of the algorithm.

\begin{tabular}{l l}
 & Description \\
\hline\hline
$\Delta$ & Maximum degree of $3$-graph \\
$\Delta_2$ & Maximum degree of $2$-graph \\
$\delta$ & Maximum codegree \\
$\omega$ & Color bound, tending to $\infty$ with $\Delta$ \\
$\epsilon$ & Small constant \\
$\omega_0$ & Error term depending on $H$ \\
$\hat{p}$ & Threshold probability \\ \hline
\end{tabular}

\begin{tabular}{l l l}
 & Value & Description \\
\hline\hline
$C$ & $\sqrt{\Delta}/\sqrt{\omega}$ & Number of colors \\
$T$ & $(5\omega/\epsilon) \log \omega$ & Number of iterations \\
$\theta$ & $\epsilon/\omega$ & Activation probability \\
$m$ & $21$ & Used to control codegrees \\
$\omega_1$ & $T\log C$ & Error term \\
$\omega_2$ & $\omega_0 / 16\omega$ & Error term \\
$\omega_3$ & $\omega^2$ & Error term \\
$\omega_4$ & $\omega^2$ & Error term \\
$\omega_5$ & $\Delta^{19/20}$ & Error term \\
$\omega_6$ & $\Delta^{1/4}$ & Error term \\
\hline
\end{tabular}

\noindent
We will use the following notation:
\begin{align*}
N^i_H(u) &= \{v \in V(H^i)-u: \exists e \in H^i \text{ with } u,v \in e\} \\
N^i_H(u,v) &= \{w \in V(H^i)-\{u,v\}: \{u,v,w\} \in H^i \} \\
N^i_c(u) &= \{v \in V(G^i_c) -u: \exists e \in G^i_c \text{ with } u,v \in e\} \\
N^i(u) &= N_H^i(u) \cup \cup_c N^i_c(u) \\
N^0_G(u) &= \{v \in V(G): uv \in E(G)\} \\
d^i_H(u) &= |\{e \in H^i: u \in e\}| \\
d^i_H(u,v) &= |\{e \in H^i: u,v \in e\}| \\
d^i_{G_c}(u) &= |\{v \in G^i_c: uv \in G_c\}|.
\end{align*}
At the beginning of iteration $i$ of the algorithm, we also define the
following parameters:
\begin{align*}
w(p_u^i) &= \sum_c p_u^i(c) \\
f_u^i(c) &= \sum_{uv \in G^i_c}p_v^i(c) \\
f_u^i &= \sum_c \sum_{uv \in G^i_c} p_u^i(c)p_v^i(c) \\
e_{uvw}^i &= \sum_c p_u^i(c)p_v^i(c)p_w^i(c) \\
e_u^i & = \sum_{uvw \in H^i} e_{uvw}^i \\
e_u^i(c) & = \sum_{uvw \in H^i} p_v^i(c)p_w^i(c) \\
h_u^i & = -\sum_c p_u^i(c)\log p_u^i(c), \text{ where $x\log x := 0$ if $x=0$ }.
\end{align*}

\noindent
Our analysis assumes that the parameters of the algorithm satisfy the following relations.
All asymptotic notation assumes $\Delta \to \infty$.
\begin{enumerate}[  ({R}1)]
\item $\theta \log(\hat{p}C) \geq 85$
\item $1/\omega_0 = o(\theta)$
\item $2/\omega_1^2 C \hat{p}^2 > 6 \log{\Delta}$
\item $T/\omega_1 = o(1)$
\item $(T\log C) / \omega_1 < \epsilon/\theta$

\item $2/(4\Delta^2 \omega_2^2 C \hat{p}^6) > 6 \log{\Delta}$
\item $\theta T / \omega_2 = o(1)$
\item $\omega \omega_2 + T < \omega_0 /2$

\item $1/\omega_2 \leq (1-\theta/4)^T \omega$

\item $1/(4\omega_3^2(6\omega_6T\theta\hat{p}^5\Delta^2 + 4m\hat{p}^5\Delta^{2+1/2m} + Cm^2\hat{p}^6\Delta^{2+1/m})) \geq 7\log\Delta$
\item $2/(4\omega_3^2C(m\Delta^{1+1/2m}\hat{p}^3 + \delta\Delta^{1/2+1/2m}\hat{p}^3)^2) \geq 7\log\Delta$

\item $2/(\omega_4^2 C (-\hat{p}\log\hat{p})^2) > 6\log{\Delta}$
\item $1/\omega_4 \leq \epsilon(1-\theta/4)^T$

\item $2\omega_5^2/ (C(m\Delta^{1+1/2m}\hat{p} + \Delta^{1/2 + 1/2m}\hat{p}\delta)^2) \geq 7\log\Delta$
\item $\omega_5 < (\theta/6)(1-\theta/3)^T\Delta$

\item $\omega_6 \Delta \theta  \hat{p} / (5\delta) \geq 6\log\Delta$
\item $\theta \omega (1-\theta/4)^T \geq \theta T / \omega_2 + 1/\omega_3$
\item $1-10\epsilon \geq 3/4$
\item $\Delta_2 \leq \omega_6\theta\Delta\hat{p}$

\item $\Delta_2 \leq \sqrt{\Delta}\sqrt{\omega}$
\item $\hat{p} \geq \Delta^{-1/2}$.
\end{enumerate}
\noindent
The analysis in Section \ref{sec_analysis} only requires that
\eqref{prcnotinuvw}, \eqref{prcnotinuvh}, \eqref{prcnotinuv}, and (R1)-(R21) hold;
the parameters $\omega$, $\epsilon$, $\hat{p}$, and $\omega_0$ 
depend on the structure of the hypergraph.
For instance, we will use the following bounds when applying the analysis to
triangle-free hypergraphs.
\begin{claim}\label{paramsclaim}
The following inequalities are consistent, and if they hold, then (R1)-(R21) also hold:
\begin{align*}
\epsilon &\leq 1/40 & \Delta_2 &\leq \sqrt{\Delta}\sqrt{\omega}  \\
\omega &< (1/26)(\epsilon/86)\log \Delta & \delta &\leq \Delta^{6/10} \\
\omega_0 &> 10\omega^3\log\omega & \hat{p} &> e^{86\omega/\epsilon}\sqrt{\omega}/\sqrt{\Delta} \\
& & \hat{p} &\leq \Delta^{-11/24}. \\
\end{align*}
\end{claim}
\begin{proof}
The bounds on $\omega$ and $\epsilon$ imply
\[
\frac{e^{86\omega/\epsilon}\sqrt{\omega}}{\sqrt{\Delta}}
< \Delta^{1/26 - 1/2}\sqrt{\omega}
= \Delta^{-6/13}\sqrt{\omega}
\leq \Delta^{-11/24},
\]
so the inequalities are consistent.
Checking that they satisfy (R1)-(R21) (for $\Delta$ sufficiently large) is straightforward.
\end{proof}

\section{Analysis of Algorithm}\label{sec_analysis}
\begin{thm}\label{analysisthm}
If \eqref{prcnotinuvw}, \eqref{prcnotinuvh}, \eqref{prcnotinuv}, and (R1)-(R21) 
hold and $|C(u)| \leq C$ for all vertices $u$, 
then the algorithm produces a proper list coloring of $H \cup G$.
\end{thm}
\begin{proof}
By Lemma \ref{mainlemma}, our algorithm proceeds for $T$ iterations, coloring most of the vertices.
Since Lemmas \ref{mainlemma}, \ref{auxlemma} and \ref{bbound} hold after iteration $T$,
we may color the remaining vertices as described in Section \ref{sec_finalstep}.
\end{proof}

\begin{lemma}[Main Lemma]\label{mainlemma}
If \eqref{prcnotinuvw}, \eqref{prcnotinuvh}, \eqref{prcnotinuv}, and (R1)-(R21) hold, 
then for each $i=0,1,\dots\,T$, the following properties hold:
\begin{enumerate}[({P}1)]
\item $|1 - w(p^i_u)| \leq i/\omega_1$.
\item $e^i_u \leq (1-\theta/3)^i\omega + i/\omega_2$
\item $f^i_u \leq 8(1-\theta/4)^i\omega$
\item $h^i_u \geq h^0_u - 21\epsilon \sum_{j=0}^{i-1}(1-\theta/4)^j$
\item $d_H^i(u) \leq (1-\theta/3)^i \Delta$
\item $d^i_{G_c}(u) \leq 3 \omega_6 i\theta\Delta\hat{p}$.
\end{enumerate}
\end{lemma}

\noindent
The proof of the Main Lemma relies on the next three lemmas.

\begin{lemma}\label{auxlemma}
For any $i=0,1,\dots\,T-1$, if 
\eqref{prcnotinuvw}, \eqref{prcnotinuvh}, \eqref{prcnotinuv}, and (R1)-(R21) hold and
$|B^i(u)| \leq \epsilon/\hat{p}$ for all $u \in U^i$, 
then there is an assignment of colors to the vertices in $U^i$ 
so that the following properties hold:
\begin{enumerate}[({Q}1)]
  \item \label{q1} $|w(p_u^{i+1})-w(p_u^i)| \leq 1/\omega_1$
  \item \label{q2} $e^{i+1}_{uvw} \leq e^{i}_{uvw} + 1/(\Delta \omega_2)$
  \item \label{q3} $f^{i+1}_u \leq f^i_u(1-\theta/2) + \theta e^i_u + 1/\omega_3$
  \item \label{q4} 
    $h^{i}_u - h_u^{i+1} \leq 2\theta(f^i_u + e^i_u) + 1/\omega_4$
  \item \label{q5} $d_H^{i+1}(u) \leq (1-\theta/2)d_H^i(u) + \omega_5$
  \item \label{q6} $d^{i+1}_{G_c}(u) \leq d^i_{G_c}(u) + 2\omega_6\theta\Delta \hat{p}$.
\end{enumerate}
\end{lemma}

\begin{lemma}\label{q2plemma}
If (Q1)-(Q6) hold for $i$ and (P1)-(P6) hold for $i$, 
then (P1)-(P6) hold for $i+1$.
\end{lemma}

\begin{lemma}\label{bbound}
If (P1)-(P6) hold for $i+1$ and (R1) and (R5) hold, then $|B^{i+1}(u)| \leq \epsilon/\hat{p}$.
\end{lemma}

\subsection{Proof of Main Lemma}
The proof relies on Lemmas \ref{auxlemma}, \ref{q2plemma} and \ref{bbound}.
Assuming these lemmas, we proceed inductively as follows:
properties (P1)-(P6) hold for $i = 0$ ((P3) holds by (R20)). Assume (P1)-(P6) hold for $i$.
By Lemma \ref{bbound}, $|B^{i}(u)| \leq \epsilon/\hat{p}$, so by Lemma \ref{auxlemma},
(Q1)-(Q6) hold for $i$. Thus Lemma \ref{q2plemma} implies (P1)-(P6) hold for $i+1$.

\subsection{Proof of Lemma \ref{q2plemma}}

\noindent \textbf{Proof of (P1)}. By (P1) (for $i$) and (Q1),
\begin{align*}
|1-w(p^{i+1}_u)| &= |1-w(p^i_u)+w(p^i_u)-w(p^{i+1}_u)| \\
&\leq |1-w(p^i_u)| + |w(p^{i+1}_u)-w(p^i_u)| \\
&\leq (i+1)/\omega_1.
\end{align*}

\noindent \textbf{Proof of (P5)}. Using (P5) (for $i$),
\begin{align*}
d_H^{i+1}(u) 
\comp{\leq}{(Q5)} (1-\theta/2)d_H^i(u) + \omega_5
&\comp{\leq}{(P5)} (1-\theta/2)(1-\theta/3)^i\Delta + \omega_5 \\
&= (1-\theta/3)^{i+1}\Delta - \frac{\theta}{6}(1-\theta/3)^i\Delta + \omega_5 \\
&\leq (1-\theta/3)^{i+1}\Delta - \frac{\theta}{6}(1-\theta/3)^T\Delta + \omega_5 \\
&\comp{\leq}{(R15)} (1-\theta/3)^{i+1}\Delta.
\end{align*}

\noindent \textbf{Proof of (P2)}. By (Q2), 
\[
e^{i+1}_{uvw} \leq e^0_{uvw} + (i+1)/\Delta\omega_2 
 \leq C(1/C^3) + (i+1)/\Delta\omega_2 = \omega/\Delta + (i+1)/\Delta\omega_2.
\]
So by (P5) (for $i+1$),
\[
e^{i+1}_u = \sum_{uvw}e^{i+1}_{uvw} \leq (1-\theta/3)^{i+1} \Delta(\omega/\Delta + (i+1)/\Delta \omega_2)
\leq (1-\theta/3)^{i+1}\omega + (i+1)/\omega_2.
\]

\noindent \textbf{Proof of (P3)}.
By (P3) and (P2) (for $i$),
\begin{align*}
f^{i+1}_u 
&\comp{\leq}{(Q3)} f^i_u (1-\theta/2) + \theta e^i_u + 1/\omega_3 \\
&\comp{\leq}{(P3)} 8(1-\theta/4)^i \omega (1-\theta/2) + \theta e^i_u + 1/\omega_3 \\
&\comp{\leq}{(P2)} 8(1-\theta/4)^i \omega (1-\theta/2) + \theta \omega (1-\theta/3)^i + \theta T/\omega_2 + 1/\omega_3 \\
&= 8(1-\theta/4)^i \omega (1-\theta/4 - \theta/4) + \theta \omega (1-\theta/3)^i + \theta T/\omega_2 + 1/\omega_3 \\
&= 8(1-\theta/4)^{i+1} \omega - 2\theta\omega(1-\theta/4)^i + \theta \omega (1-\theta/3)^i + \theta T/ \omega_2 + 1/\omega_3 \\
&< 8(1-\theta/4)^{i+1} \omega - \theta\omega(1-\theta/4)^i + \theta T / \omega_2 + 1/\omega_3 \\
&\comp{\leq}{(R17)} 8(1-\theta/4)^{i+1}\omega.
\end{align*}

\noindent \textbf{Proof of (P4)}. 
We have
\begin{equation}\label{tomega}
T/\omega_2 \comp{\leq}{(R17)} \omega(1-\theta/4)^T \leq \omega(1-\theta/4)^i.
\end{equation}
Therefore, using $\epsilon=\omega\theta$ and (P4) (for $i$),
\begin{align*}
h^{i+1}_u 
&\comp{\geq}{(Q4)} h^i_u - 2\theta(f^i_u+e^i_u) - 1/\omega_4 \\
&\comp{\geq}{(P3)} h^i_u - 2\theta(8(1-\theta/4)^i\omega 
 + e^i_u) - 1/\omega_4 \\
&\comp{\geq}{(P2)} h^i_u - 2\theta(8(1-\theta/4)^i\omega 
 + (1-\theta/3)^i\omega+T/\omega_2) - 1/\omega_4 \\
&\geq h^i_u - 2\theta(9(1-\theta/4)^i\omega +T/\omega_2) - 1/\omega_4 \\
&\comp{\geq}{\eqref{tomega}} h^i_u - 2\theta(10(1-\theta/4)^i\omega) - 1/\omega_4 \\
&= h^i_u - 20\epsilon(1-\theta/4)^i - 1/\omega_4 \\
&\comp{\geq}{(R13)} h^i_u - 21\epsilon(1-\theta/4)^i \\
&\comp{\geq}{(P4)} h^0_u - 21\epsilon \sum_{j=0}^{i-1}(1-\theta/4)^j - 
  21\epsilon(1-\theta/4)^i \\
&= h^0_u - 21\epsilon\sum_{j=0}^{i}(1-\theta/4)^j.
\end{align*}

\noindent \textbf{Proof of (P6)}. By (Q6) and (R19),
\begin{align*}
d^{i+1}_{G_c}(u) \comp{\leq}{(Q6)} 
\Delta_2 + 2\omega_6(i+1)\theta\Delta\hat{p} \comp{\leq}{(R19)} 3\omega_6(i+1)\theta\Delta\hat{p}.
\end{align*}

\subsection{Proof of Lemma \ref{bbound}}
\noindent
First, 
\begin{align}\label{bboundeq1}
|B^{i+1}(u)|\hat{p}\log(\hat{p}C) 
= \sum_{c \in B^{i+1}(u)}\hat{p}\log(\hat{p}C)
&= \sum_{c \in B^{i+1}(u)}p^{i+1}_u(c)\log(p^{i+1}_u(c)C) \notag \\
&\leq \sum_{c \in C(u) }p^{i+1}_u(c)\log(p^{i+1}_u(c)C) \notag \\
&= \sum_{c \in C(u)} p^{i+1}_u(c) \log p^{i+1}_u(c) + \sum_{c \in C(u)} p^{i+1}_u(c)\log C \notag \\
&= -h^{i+1}_u + \log C \sum_{c \in C(u)} p^{i+1}_u(c).
\end{align}
Using $p^0_u(c) = 1/C$ for all $c \in C(u)$,
\begin{align*}
h^0_u
&= -\sum_{c \in C(u)} p_u^0(c)\log p_u^0(c) \\
&= \log C\sum_{c \in C(u)} p_u^0(c) \\
&= \log C\sum_{c \in C(u)} (p_u^0(c)-p_u^{i+1}(c)) + \log C \sum_{c\in C(u)} p_u^{i+1}(c) \\
&= \log C (1-w(p_u^{i+1}))  +\log C \sum_{c\in C(u)} p_u^{i+1}(c) \\
&\comp{\geq}{(P1)} -(T\log C)/\omega_1 + \log C \sum_{c\in C(u)} p_u^{i+1}(c) \\
&\comp{>}{(R5)} -\epsilon/\theta + \log C \sum_{c\in C(u)} p_u^{i+1}(c).
\end{align*}
Using $\sum_{j=0}^i (1-\theta/4)^j \leq 4/\theta$,
the above inequality, and inequality \eqref{bboundeq1},
\begin{align*}
h^{i+1}_u \comp{\geq}{(P4)} h^0_u - 21\epsilon \sum_{j=0}^i (1-\theta/4)^j
 \geq h^0_u - 84\epsilon/\theta
 &\geq \log C \sum_{c\in C(u)} p^{i+1}_u(c) - 85\epsilon/\theta \\
 &\comp{\geq}{\eqref{bboundeq1}} h^{i+1}_u + |B^{i+1}(u)|\hat{p}\log(\hat{p}C) - 85\epsilon/\theta.
\end{align*}
So
\begin{align*}
|B^{i+1}(u)| \leq \frac{85 \epsilon}{\theta \hat{p} \log(\hat{p}C)}
\comp{\leq}{(R1)} \epsilon/\hat{p}.
\end{align*}

\subsection{Proof of Lemma \ref{auxlemma}}
We are going to apply the Local Lemma.
Our probability space is determined by coin flips at each vertex
which determine the random variables $\gamma_u(c)$ and $\eta_u(c)$.
The random variable $p_u(c)$ is determined by the coin flips in $N(u)$.
The events ``(Q1) fails to hold for $u$'' and ``(Q4) fails to hold for $u$''
are therefore determined by these coin flips.
The events ``(Q3) fails to hold for $u$'' and ``(Q5) fails to hold for $u$''
are determined by the coin flips in $N(N(u))$.
The event ``(Q2) fails to hold for edge $uvw$'' is determined by the coin flips in
$N(N(u)) + N(N(v)) + N(N(w))$.
The event ``(Q6) fails to hold for $u$ and $c$'' is determined by the coin flips
in $N(N(u))$.
Each event is therefore mutually independent of at most $5(\Delta+\Delta_2)^4$ 
(Q1), (Q3), (Q4), (Q5), or (Q6) events and at most $\Delta(3\Delta+3\Delta_2)^4$ 
(Q3) events. By (R20), $\Delta_2 < \Delta$, so 
each event is mutually independent of at most $7^4\Delta^5$ other events.

It therefore suffices to show that the probability that (Q$i$) fails
 is less than $4(7^4)\Delta^{-5}$. 
We prove this for (Q1), (Q2), (Q4), and (Q6) first,
and then move on to (Q3) and (Q5). Throughout the proof, we drop the notation $i+1$ and $i$,
and use, for instance, $p'_u(c)$ and $p_u(c)$ to denote values in iterations $i+1$ and $i$,
respectively.

\noindent \textbf{Proof of (Q1)}.
By \eqref{expuc}, $\Ex[p'_u(c)] = p_u(c)$ for each color $c$.
By linearity of expectation,
\[
\Ex[w(p'_u)] = w(p_u).
\]
Since $w(p'_u)$ is the sum of $C$ independent
non-negative random variables, each bounded by $\hat{p}$, Theorem \ref{hoeffding}
and (R3) imply
\begin{align*}
\Pr[|w(p'_u)-w(p_u)| \geq 1/\omega_1] 
\leq 2e^{- 2/(C \hat{p}^2 \omega_1^2) }
&< 2e^{- 6\log\Delta }.
\end{align*}

\noindent \textbf{Proof of (Q2)}.
Suppose $uvw \in H$. We first prove
\begin{equation}\label{exuvw}
\Ex[p'_u(c)p'_v(c)p'_w(c)] \leq p_u(c)p_v(c)p_w(c)(1+1/\omega_0).
\end{equation}
Assume that $p'_u(c)$, $p'_v(c)$, and $p'_w(c)$ are determined by \eqref{flip1}.
If $c \in L(u) \cup L(v) \cup L(w)$, then $p'_u(c)p'_v(c)p'_w(c) = 0$,
so by \eqref{prcnotinuvw},
\begin{align*}
\Ex[p'_u(c)p'_v(c)p'_w(c)]
&\leq \frac{p_u(c)}{q_u(c)}\frac{p_v(c)}{q_v(c)}\frac{p_w(c)}{q_w(c)}
\Pr[c \notin L(u) \cup L(v) \cup L(w)] \\
&\leq p_u(c)p_v(c)p_w(c)(1+1/\omega_0).
\end{align*}
Suppose $p'_u(c)$ and $p'_v(c)$ are determined by \eqref{flip1}, and
$p'_w(c)$ is determined by \eqref{flip2}. Then
$p'_w(c)$ is independent of $p'_u(c)$ and $p'_v(c)$, so by \eqref{prcnotinuvh},
\begin{align*}
\Ex[p'_u(c)p'_v(c)p'_w(c)]
&= \Ex[p'_u(c)p'_v(c)]\Ex[p'_w(c)] \\
&\leq \frac{p_u(c)}{q_u(c)}\frac{p_v(c)}{q_v(c)}\Pr[c \notin L(u) \cup L(v)] p_w(c) \\
&\leq p_u(c)p_v(c)p_w(c)(1+1/\omega_0).
\end{align*}
If at least two of $p'_u(c)$, $p'_v(c)$, and $p'_w(c)$ are determined by \eqref{flip2}, then
all three are independent of each other, and 
\[
\Ex[p'_u(c)p'_v(c)p'_w(c)] = p_u(c)p_v(c)p_w(c),
\]
finishing the proof of \eqref{exuvw}.

By definition, $e^0_{uvw} \leq C/C^3 = \omega/\Delta$. So by (Q2) (for $i$) and (R8), 
\[
e_{uvw} / \omega_0
\comp{\leq}{(Q2)} (e^0_{uvw} + \frac{i}{\Delta \omega_2}) \frac{1}{\omega_0}
\leq (\frac{\omega}{\Delta} + \frac{T}{\Delta \omega_2}) \frac{1}{\omega_0}
= \frac{\omega \omega_2 + T}{\omega_0} \frac{1}{\Delta \omega_2}
\comp{<}{(R8)} 1/(2\Delta \omega_2).
\]
So by \eqref{exuvw},
\begin{align*}
\Ex[e'_{uvw}] = \sum_c \Ex[p'_u(c)p'_v(c)p'_w(c)] 
&\leq \sum_c p_u(c)p_v(c)p_w(c)(1+1/\omega_0) \\
&= e_{uvw}(1+1/\omega_0) \\
&< e_{uvw}+1/(2\Delta\omega_2).
\end{align*}
Now $e'_{uvw}$ is the sum of $C$ independent random variables,
each bounded by $\hat{p}^3$. 
Thus Theorem \ref{hoeffding} and (R6) yield
\begin{align*}
\Pr[e'_{uvw} \geq e_{uww} + 1/(\Delta \omega_2)]
&\leq \Pr[e'_{uvw} \geq e_{uvw} + 1/(2\Delta\omega_2) + 1/(2\Delta \omega_2 )] \\
&\leq \Pr[e'_{uvw} \geq \Ex[e'_{uvw}] + 1/(2\Delta \omega_2 )] \\
&< e^{-2/(4 \Delta^2 \omega_2^2 C\hat{p}^6)} \\
&< e^{-6\log\Delta}.
\end{align*}

\noindent \textbf{Proof of (Q4)}.
By \eqref{flip1} and \eqref{flip2}, $p'_u(c) = p_u(c)\I[A]/\Pr[A]$ for some event $A$.
Thus, using $x\log x = 0$ for $x \in \{0,1\}$,
\begin{align*}
\Ex[p'_u(c)\log p'_u(c)] 
&= \Ex[p_u(c)\I[A] / \Pr[A] \log (p_u(c)\I[A]/\Pr[A])] \\
&= \Ex[p_u(c)\I[A]/\Pr[A] \log{p_u(c)} + p_u(c)\I[A]/\Pr[A]\log{(\I[A]/\Pr[A])}] \\
&= \frac{p_u(c)\log{p_u(c)}}{\Pr[A]}\Ex[\I[A]] + \frac{p_u(c)}{\Pr[A]}\Ex[\I[A]\log{(\I[A]/\Pr[A])}] \\
&= p_u(c)\log{p_u(c)} + \frac{p_u(c)}{\Pr[A]}\Ex[\I[A]\log{\I[A]}] - \frac{p_u(c)}{\Pr[A]}\Ex[\I[A]\log{\Pr[A]}] \\
&= p_u(c)\log{p_u(c)} + \frac{p_u(c)}{\Pr[A]}\Ex[0] - p_u(c)\log{\Pr[A]} \\
&= p_u(c)\log{p_u(c)} - p_u(c)\log{\Pr[A]}.
\end{align*}
Recall that
\[
q_u(c) =
\Pr[ \bigcap_{uvw \in H} (\gamma_v(c)=0 \cup \gamma_w(c)=0)
  \bigcap_{uv \in G_c} \gamma_v(c)=0 ].
\]
Also, $1-rx \geq (1-x)^r$ for  $r,x \in (0,1)$.
Finally, the event $\gamma_v(c)=0$ is monotone decreasing, so by the FKG inequality,
\begin{align*}
q_u(c) 
&\comp{\geq}{FKG}
\prod_{uvw \in H} \Pr[\gamma_v(c)=0 \cup \gamma_w(c)=0]
\prod_{uv \in G_c} \Pr[\gamma_v(c)=0 ] \\
&=
\prod_{uvw \in H} (1-\theta^2 p_v(c)p_w(c)) \prod_{uv \in G_c} (1-\theta p_v(c)) \\
&\geq
\prod_{uvw \in H} (1-\theta)^{\theta p_v(c)p_w(c)} \prod_{uv \in G_c} (1-\theta)^{p_v(c)}.
\end{align*}
By the algorithm, $\Pr[A] \geq q_u(c)$. 
Also, $\log(1-x) \geq -x - x^2$ for $x \in [0,1/3]$.
Combining these inequalities with the previous inequality, we obtain
\begin{align*}
\log{\Pr[A]} \geq \log{q_u(c)} 
&\geq \log{(\prod_{uvw \in H} (1-\theta)^{\theta p_v(c)p_w(c)} \prod_{uv \in G_c} (1-\theta)^{p_v(c)})} \\
&= \sum_{uvw \in H} \theta p_v(c)p_w(c)\log (1-\theta) + \sum_{uv \in G_c} p_v(c)\log(1-\theta) \\&\geq \sum_{uvw \in H} \theta p_v(c)p_w(c)(-\theta-\theta^2) + \sum_{uv \in G_c} p_v(c)(-\theta-\theta^2) \\
&= (-\theta^2-\theta^3)\sum_{uvw \in H} p_v(c)p_w(c) + (-\theta-\theta^2)\sum_{uv \in G_c}p_v(c) \\
&= -(\theta^2+\theta^3)e_u(c) - (\theta+\theta^2)f_u(c).
\end{align*}
Therefore, using the definition of $h_u$ and $\theta < 1/2$,
\begin{align*}
\Ex[h_u - h'_u] 
&= h_u + \sum_c \Ex[p'_u(c)\log p'_u(c)] \\
&= h_u + \sum_c p_u(c)\log p_u(c) - \sum_c p_u(c)\log\Pr[A]) \\
&= - \sum_c p_u(c)\log\Pr[A] \\
&\leq \sum_c p_u(c)((\theta+\theta^2)f_u(c)+(\theta^2+\theta^3)e_u(c)) \\
&= (\theta+\theta^2)f_u + (\theta^2 + \theta^3)e_u \\
&< 2\theta(f_u + e_u).
\end{align*}
The terms in $\sum_c -p'_u(c)\log p'_u(c)$ are independent and,
since $-x\log x$ is increasing for $0 < x \leq \hat{p}$,
bounded by $-\hat{p}\log{\hat{p}}$.
Thus, by Theorem \ref{hoeffding} and (R12),
\begin{align*}
\Pr[h_u-h'_u \geq 2\theta(f_u + e_u) + 1/\omega_4]
< e^{-2/(\omega_4^2 C (-\hat{p}\log\hat{p})^2)}
< e^{-6\log{\Delta}}.
\end{align*}

\noindent \textbf{Proof of (Q6)}.
Fix $c \in C(u)$. For each $v \in N_H(u)$, set
\[
X_v = d_H(u,v)\gamma_v(c),
\]
and set
\[
X = \sum_{v \in N_H(u)}X_v.
\]
Then
\[
\Ex[X] = \sum_{v \in N_H(u)}d_H(u,v)p_v(c)\theta \leq  \hat{p}\theta \sum_{v \in N_H(u)}d_H(u,v) \leq 2\Delta\hat{p}\theta.
\]
Since the $X_v$ are independent from each other (because the $\gamma_v(c)$ are independent), 
and $x(1-x)$ is increasing for $x < 1/2$,
\begin{align*}
\Var[X] = \sum_{v \in N_H(u)}\Var[X_v] 
&= \sum_{v \in N_H(u)}(\Ex[X_v^2]-\Ex[X_v]^2) \\
&= \sum_{v \in N_H(u)}(d_H(u,v)^2p_v(c)\theta - d_H(u,v)^2p_v(c)^2\theta^2) \\
&\leq \sum_{v \in N_H(u)}d_H(u,v)^2\hat{p}\theta(1- \hat{p}\theta) \\
&= \hat{p}\theta(1-\hat{p}\theta)\sum_{v \in N_H(u)}d_H(u,v)^2 \\
&\leq \hat{p}\theta(1-\hat{p}\theta)\delta \sum_{v \in N_H(u)}d_H(u,v) \\
&= \hat{p}\theta(1-\hat{p}\theta)2 \Delta \delta \\
&< \hat{p}\theta2 \Delta \delta.
\end{align*}
If $uv \notin G_c$ and $uv \in G'_c$, then there exists an edge $uvw \in H$ such that
$\gamma_w(c)=1$. Hence
\[
d'_{G_c}(u)-d_{G_c}(u) \leq 
\sum_{uvw \in H}(\gamma_v(c) + \gamma_w(c)) =
\sum_{v \in N_H(u)} d_H(u,v)\gamma_v(c) =
X.
\]
Applying Theorem \ref{hoeffdingvar} (with $b = \delta$) and (R16),
\begin{align*}
\Pr[d'_{G_c}(u)-d_{G_c}(u) \geq 2\omega_6\Delta\hat{p}\theta]
& \leq \Pr[X \geq \omega_6 \Delta\hat{p}\theta + \omega_6 \Delta\hat{p}\theta] \\
& \leq \Pr[X \geq \Ex[X] + \omega_6 \Delta\hat{p}\theta] \\
& \leq e^{-\omega_6^2 \Delta^2 \hat{p}^2 \theta^2 / (4\hat{p}\theta\Delta \delta + \delta \omega_6  \Delta \hat{p}\theta)} \\
& \leq e^{-\omega_6^2 \Delta^2 \hat{p}^2 \theta^2 / 5 \delta \omega_6  \Delta \hat{p}\theta} \\
& < e^{-\omega_6 \Delta \hat{p} \theta / 5\delta} \\
& \comp{<}{(R16)} e^{-6\log\Delta}.
\end{align*}

\noindent
We now prove (Q3) and (Q5).
The following two claims will be used in both proofs.
\begin{claim}\label{prcol}
For any $v \in U$ and $c \in C(v)$, 
\[
\Pr[v \notin U' | c \notin L(v)] \geq \Pr[v \notin U'] \geq 3\theta/4,
\]
and if $uv \in G_c$, then
\begin{align*}
\Pr[v \notin U' | c \notin L(u)] &\geq \Pr[v \notin U'] - \theta\hat{p} \geq 5\theta/8, \\
\Pr[v \notin U' | c \notin L(u) \cup L(v)] &\geq \Pr[v \notin U'] - \theta\hat{p} \geq 5\theta/8.
\end{align*}
\end{claim}
\begin{proof}[Proof of claim]
The vertex $v$ is colored (i.e., $v \notin U'$) if and only if for some color $d \notin B(v)$,
$\gamma_v(d)=1$ and $d \notin L(v)$. Let $R_d$ denote the event that
$\gamma_v(d)=1$ and $d \notin L(v)$.
If $c \in B(v)$, then $v$ cannot be colored $c$, so the event
$v \notin U'$ is independent of the events $c \notin L(v)$ and $c \notin L(u)$;
hence
\[
\Pr[v \notin U'] 
= \Pr[v \notin U' | c \notin L(v)] = \Pr[v \notin U' | c \notin L(u)] 
= \Pr[v \notin U' | c \notin L(u) \cup L(v)].
\]
Otherwise, 
\begin{align*}
\Pr[v \notin U' | c \notin L(v)]
&= \frac{\Pr[v \notin U', c \notin L(v)]}{\Pr[c \notin L(v)]} \\
&= \frac{\Pr[\cup_{d \notin B(u)}R_d, c \notin L(v)]}{\Pr[c \notin L(v)]} \\
&= \frac{\Pr[(\cup_{d \notin B(u)+c}R_d \cup R_c), c \notin L(v)]}{\Pr[c \notin L(v)]} \\
&= \frac{\Pr[(\cup_{d \notin B(u)+c}R_d \cup \gamma_v(c)=1), c \notin L(v)]}{\Pr[c \notin L(v)]} \\
&= \frac{\Pr[(\cup_{d \notin B(u)+c}R_d \cup \gamma_v(c)=1)]\Pr[c \notin L(v)]}{\Pr[c \notin L(v)]} \\
&= \Pr[(\cup_{d \notin B(v)+c}R_d) \cup (\gamma_v(c) = 1)] \\
&\geq \Pr[(\cup_{d \notin B(v)+c}R_d) \cup R_c] \\
&= \Pr[v \notin U'].
\end{align*}

Suppose $uv \in G_c$. If $c \notin L(u)$, then $\gamma_w(c)=0$ for all $w \in N_{G_c}(u)$, so in particular, $\gamma_v(c)=0$. Consequently, 
\[
\Pr[R_c | c \notin L(u) \cup L(v)] = \Pr[\gamma_v(c)=1 \cap c \notin L(v) | c \notin L(u) \cup L(v)] = 0.
\]
So by the independence of colors and the inequality
\[
\Pr[\cup_{d \in C(v)-B(v)} R_d] \leq \Pr[\cup_{d \in C(v)-B(v)-c} R_d] + \Pr[R_c],
\]
we obtain
\begin{align*}
\Pr[v \notin U' | c \notin L(u) \cup L(v)]
&= \Pr[\cup_{d \notin B(v)} R_d | c \notin L(u) \cup L(v)] \\
&= \Pr[\cup_{d \notin B(v)+c} R_d] \\
&= \Pr[\cup_{d \in C(v)-B(v)-c} R_d] \\
&\geq \Pr[\cup_{d \in C(v)-B(v)} R_d] - \Pr[R_c] \\
&\geq \Pr[v \notin U'] - \theta\hat{p}.
\end{align*}
Since we only used the condition $c \notin L(u)$, 
this also implies 
\[
\Pr[v \notin U' | c \notin L(u)] \geq \Pr[v \notin U'] - \theta\hat{p}.
\]
To finish the proof of the claim, 
we now show $\Pr[v \notin U']  \geq 3\theta/4$.
First,
\begin{align*}
\Pr[v \notin U'] 
&= \Pr[\cup_{d \notin B(v)}R_d] \\
&\geq \sum_{d \notin B(v)}\Pr[R_d] - \sum_{d,d' \notin B(v)}\Pr[R_d]\Pr[R_{d'}] \\
&= \sum_{d \notin B(v)}\theta p_v(d)q_v(d) - \sum_{d,d' \notin B(v)}\theta^2 p_v(d)p_v(d')q_v(d)q_v(d') \\
&\geq \theta \sum_{d\in C(v)} p_v(d)q_v(d) - \theta \sum_{d\in B(v)}p_v(d)q_v(d) - \theta^2\sum_{d,d' \notin B(v)}p_v(d)p_v(d') \\
&\geq \theta \sum_{d\in C(v)} p_v(d)q_v(d) - \theta |B(v)|\hat{p} - \theta^2\sum_{d,d' \notin B(v)}p_v(d)p_v(d').
\end{align*}
By \eqref{qbound},
\begin{align*}
q_v(d)
&\geq 1-\sum_{uvw \in H} \theta^2 p_u(d)p_w(d) - \sum_{uv \in G_d}\theta p_u(d) \\
&= 1-\theta^2\sum_{uvw \in H} p_u(d)p_w(d) - \theta \sum_{uv \in G_d}p_u(d) \\
&= 1-\theta^2 e_v(d) - \theta f_v(d).
\end{align*}
Since $\sum_{d \in C(v)}p_v(c) \leq \sqrt{2}$ (by (P1) and (R4)),
\begin{align*}
\theta^2 \sum_{d,d' \notin B(v)}  p_v(d)p_v(d') 
\leq  \frac{1}{2}\theta^2\sum_{d\in C(v)}\sum_{d' \in C(v)-d} p_v(d)p_v(d') 
\leq \frac{1}{2}\theta^2(\sum_{d\in C}p_v(d))^2 
\leq \theta^2.
\end{align*}
By our lemma's assumption, $|B(v)| \leq \epsilon/\hat{p}$.
By (P3), $f_v < 8\omega$, so $\theta f_v < 8\epsilon$.
By (P2), $e_v \leq \omega + T/\omega_2$, so (R7) implies $\theta^2 e_v < \epsilon/3$.
Using these three inequalities, $\sum_{d\in C(v)}p_v(c) \geq (1-\epsilon/3)$, and (R18), we finally obtain
\begin{align*}
\Pr[v \notin U']
&\geq \theta \sum_{d\in C(v)} p_v(d)(1-\theta^2 e_v(d) - \theta f_v(d)) - \theta |B(v)|\hat{p} - \theta^2 \\
&= \theta \sum_{d\in C(v)} p_v(d) - \theta^3 \sum_{d\in C(v)} p_v(d)e_v(d) - \theta^2 \sum_{d\in C(v)} p_v(d)f_v(d) - \theta |B(v)|\hat{p}  - \theta^2 \notag \\
&\geq \theta \sum_{d\in C(v)} p_v(d) - \theta^3 \sum_{d\in C(v)} p_v(d)e_v(d) - \theta^2 \sum_{d\in C(v)} p_v(d)f_v(d) - \theta\epsilon  - \theta^2 \notag \\
&= \theta \sum_{d\in C(v)} p_v(d) - \theta^3 e_v - \theta^2 f_v - \theta\epsilon  - \theta^2  \\
&\geq \theta(1-\epsilon/3) - \theta \epsilon/3 - 8\theta \epsilon - \theta\epsilon - \theta \epsilon/3 \\
&= \theta(1-10\epsilon) \\
&\geq 3\theta/4.
\end{align*}
\end{proof}

\noindent
Recall that $m$ is a fixed constant.
\begin{claim}\label{conccond}
For each $l = 0, \dots, m-2$, let
\[
N^0(u,l) = \{v \in N^0_H(u)-N^0_G(u): \Delta^{l/2m} < d^0_H(u,v) \leq \Delta^{(l+1)/2m}\},
\]
and for $l = m-1$, let
\[
N^0(u,l) = \{v \in N^0_H(u): d^0_H(u,v) > \Delta^{l/2m} \} \cup N^0_G(u).
\]
For each $l$ and color $c$, 
let $\mathcal{A}_{c,l}$ be the event that $\gamma_v(c)=1$ for at most $\Delta^{1-l/2m}\hat{p}$
vertices $v \in N^0(u,l)$. 
Let $\mathcal{A}$ denote the event that $\mathcal{A}_{c,l}$ holds for all $l$ and $c$.
Then
\[
\Pr[\bar{\mathcal{A}}] \leq e^{-10\log\Delta}.
\]
\end{claim}
\begin{proof}[Proof of claim]
Suppose $l < m - 1$. Since each $v \in N^0(u,l)$ contributes at least $\Delta^{l/2m}$ edges 
to $d^0_H(u)$, and each edge is counted at most twice,
\[
|N^0(u,l)| \leq 2\Delta/\Delta^{l/2m} = 2\Delta^{1-l/2m}.
\]
If $l=m-1$, 
\[
|N^0(u,l)| \leq 2\Delta/\Delta^{l/2m} + \Delta_2 = 2\Delta^{1-l/2m} + \Delta_2 
\comp{<}{(R20)} 3\Delta^{1-l/2m}.
\]
Thus $|N^0(u,l)| < 3\Delta^{1-l/2m}$ for each $l$.

Since $\Pr[\gamma_v(c)=1] \leq \hat{p}\theta$ and $3e\theta < 1/e$,
\begin{align*}
\Pr[\bar{\mathcal{A}}_{c,l}] 
\leq \binom{|N^0(u,l)|}{\Delta^{1-l/2m}\hat{p}}(\hat{p}\theta)^{\Delta^{1-l/2m}\hat{p}}
&\leq \binom{3\Delta^{1-l/2m}}{\Delta^{1-l/2m}\hat{p}}(\hat{p}\theta)^{\Delta^{1-l/2m}\hat{p}} \\
&\leq (\frac{3e}{\hat{p}})^{\Delta^{1-l/2m}\hat{p}}(\hat{p}\theta)^{\Delta^{1-l/2m}\hat{p}} \\
&= (3e\theta)^{\Delta^{1-l/2m}\hat{p}} \\
&< e^{-\Delta^{1-l/2m}\hat{p}} \\
&\comp{\leq}{(R21)} e^{-\Delta^{(m+1)/2m} \Delta^{-1/2}} \\
&= e^{-\Delta^{1/2m}}.
\end{align*}
So by the union bound,
\[
\Pr[\bar{\mathcal{A}}] \leq C m e^{-\Delta^{1/2m}} \leq e^{-10\log\Delta}.
\]
\end{proof}

\noindent \textbf{Proof of (Q3)}.
Observe that
\begin{align*}
f'_u 
&= \sum_c \sum_{uv \in G'_c}p'_u(c)p'_v(c) \\
&= \sum_c \sum_{uv \in G_c}p'_u(c)p'_v(c)\I[uv \in G'_c] 
+ \sum_c \sum_{\substack{uv \notin G_c \\ uv \in G'_c}} p'_u(c)p'_v(c) \\
&\leq \sum_c \sum_{uv \in G_c}p'_u(c)p'_v(c)\I[v \in U'] \\
&+ \sum_c \sum_{\substack{uvw \in H}} (p'_u(c)p'_v(c)\I[\gamma_w(c)=1]+p'_u(c)p'_w(c)\I[\gamma_v(c)=1]) \\
&= D_1 + D_2,
\end{align*}
where
\[
D_1 = \sum_c \sum_{uv \in G_c}p'_u(c)p'_v(c)\I[v \in U'],
\]
and
\[
D_2 = \sum_c \sum_{\substack{uvw \in H}} (p'_u(c)p'_v(c)\I[\gamma_w(c)=1]+p'_u(c)p'_w(c)\I[\gamma_v(c)=1]).
\]
To bound $D_1$, we first prove that for $uv \in G_c$,
\begin{equation}\label{exuvI}
\Ex[p'_u(c)p'_v(c)\I[v \in U']] \leq p_u(c)p_v(c)(1-9\theta/16).
\end{equation}
First assume that $p'_u(c)$ and $p'_v(c)$ are determined by \eqref{flip1}.
If $c \in L(u)\cup L(v)$, then $p'_u(c)p'_v(c)=0$, 
so using \eqref{prcnotinuv}, Claim \ref{prcol}, and then (R2),
\begin{align*}
\Ex[p'_u(c)p'_v(c)\I[v \in U']]
&= \Ex[p'_u(c)p'_v(c) | v \in U']\Pr[v \in U'] \\
&\leq \frac{p_u(c)}{q_u(c)}\frac{p_v(c)}{q_v(c)}
 \Pr[c \notin L(u) \cup L(v) | v \in U']\Pr[v \in U']  \\
&= \frac{p_u(c)}{q_u(c)}\frac{p_v(c)}{q_v(c)} 
 \Pr[v \in U' | c \notin L(u) \cup L(v)] \Pr[c \notin L(u) \cup L(v)] \\
&\comp{\leq}{\eqref{prcnotinuv}} p_u(c)p_v(c) (1+1/\omega_0)
 \Pr[v \in U' | c \notin L(u) \cup L(v)] \\
&\comp{\leq}{C.\ref{prcol}} p_u(c)p_v(c) (1+1/\omega_0) (1-5\theta/8)\\
&\comp{\leq}{(R2)} p_u(c)p_v(c)(1-9\theta/16).
\end{align*}
Suppose $p'_u(c)$ is determined by \eqref{flip1} and $p'_v(c)$ is determined by \eqref{flip2}.
Then $p'_u(c)$ and $p'_v(c)$ are independent of each other, and $p'_v(c)$ is independent of
the event $v \in U'$, so
\begin{align*}
\Ex[p'_u(c)p'_v(c)\I[v \in U']]
&= \Ex[p'_u(c)p'_v(c) | v \in U']\Pr[v \in U'] \\
&= \Ex[p'_u(c) | v \in U']\Ex[p'_v(c)]\Pr[v \in U'] \\
&\comp{\leq}{\eqref{expuc}} \Ex[p'_u(c) | v \in U']p_v(c)\Pr[v \in U'] \\
&\leq \frac{p_u(c)}{q_u(c)} 
 \Pr[c \notin L(u) | v \in U'] \Pr[v \in U']  p_v(c) \\
&= \frac{p_u(c)}{q_u(c)}
 \Pr[v \in U' | c \notin L(u)] \Pr[c \notin L(u)] p_v(c) \\
&= p_u(c)p_v(c) 
 \Pr[v \in U' | c \notin L(u)] \\
&\comp{\leq}{C.\ref{prcol}} p_u(c)p_v(c) (1+1/\omega_0) (1-5\theta/8)\\
&\comp{\leq}{(R2)} p_u(c)p_v(c)(1-9\theta/16).
\end{align*}
Similarly, if $p'_u(c)$ is determined by \eqref{flip2} and $p'_v(c)$ is determined by \eqref{flip1},
\begin{align*}
\Ex[p'_u(c)p'_v(c)\I[v \in U']]
&\leq p_u(c)p_v(c) 
 \Pr[v \in U' | c \notin L(v)] \\
&\comp{\leq}{C.\ref{prcol}} p_u(c)p_v(c) (1+1/\omega_0) (1-5\theta/8)\\
&\comp{\leq}{(R2)} p_u(c)p_v(c)(1-9\theta/16).
\end{align*}
If $p'_u(c)$ and $p'_v(c)$ are both determined by \eqref{flip2},
\begin{align*}
\Ex[p'_u(c)p'_v(c)\I[v \in U']]
&= \Ex[p'_u(c)p'_v(c)]\Pr[v \in U'] \\
&= \Ex[p'_u(c)]\Ex[p'_v(c)]\Pr[v \in U']\\
&\comp{\leq}{(C.\ref{prcol})} p_u(c)p_v(c)(1-3\theta/4) \\
& < p_u(c)p_v(c)(1-9\theta/16),
\end{align*}
concluding the proof of \eqref{exuvI}.

By \eqref{exuvI},
\begin{align*}
\Ex[D_1]
&= \sum_c\sum_{uv \in G_c}\Ex[p'_u(c)p'_v(c)\I[v \in U']] \\
&\leq \sum_c \sum_{uv \in G_c}p_u(c)p_v(c)(1-9\theta/16) \\
&= f_u(1-9\theta/16).
\end{align*}

\noindent
For $c \in C(u)$, let 
\[
T_c = \{\gamma_v(c): v \in N(N(u))\} \cup \{\eta_v(c): v \in N(N(u))\}.
\]
Then each $T_c$ is a (vector valued) random variable, and the
set of random variables $\{T_c: c \in C(u)\}$ are mutually independent
and determine the variable $D_1$. 
We will now apply Corollary \ref{bdbad} with parameters:
\begin{itemize}
 \item Independent random variables $T_c: \{c\} \to \{0,1\}^{2|N(N(u))|}$, for each $c \in C(u)$
 \item Events $\mathcal{A}_c = \cap_{l=1}^m {\mathcal{A}_{c,l}}$, for each $c \in C(u)$
       (where $\mathcal{A}_{c,l}$ is from Claim \ref{conccond})
 \item $\mathcal{A} = \prod_{c \in C(u)} \mathcal{A}_c$, for each $c \in C(u)$
       (this is the same $\mathcal{A}$ as in Claim \ref{conccond})
 \item $D_1$ (which is non-negative) in the role of $Y$
 \item $d_{G_c(u)}\hat{p}^2 + m\hat{p}^3\Delta^{1+1/2m}$ in the role of $d_c$.
\end{itemize}
Our goal is thus to bound the effect of $T_c$ on $D_1$ given that $\mathcal{A}$ holds.
Note first that
\[
D_1 = \sum_{uv \in G_c}p'_u(c)p'_v(c)\I[v \in U']
+ \sum_{l=0}^{m-1} \sum_{v \in N^0(u,l)}\I[v \in U'] \sum_{\substack{d \neq c: \\ uv \in G_d}}p'_u(d)p'_v(d).
\]
\noindent
The total effect of $T_c$ on the left hand sum is at most $d_{G_c}(u)\hat{p}^2$,
so consider the right hand sum.
The $p'_u(d)p'_v(d)$ terms are always independent of $T_c$.
Observe that if $\gamma_v(c)=0$, then $\I[v \in U']$ is also independent of $T_c$;
this is because if $\gamma_v(c)=0$, then $v$ can not be colored $c$ in the current round,
so $T_c$ has no impact on whether or not $v \in U'$. 
Thus $T_c$ only affects the term 
\[
\I[v \in U'] \sum_{\substack{d \neq c \\ uv \in G_d}}p'_u(d)p'_v(d)
\]
if $\gamma_v(c)=1$.
So given the event $\mathcal{A}_{c,l}$ from Claim \ref{conccond}, 
$T_c$ affects at most $\Delta^{1-l/2m}\hat{p}$ such terms for
each $l$.
If $v \in N^0(u,l)$, where $l \leq m-2$, the effect is at 
most $d^0_H(u,v)\hat{p}^2 \leq \Delta^{(l+1)/2m}\hat{p}^2$.
If $l = m-1$, the effect is at most $C \hat{p}^2 < \Delta^{1/2}\hat{p}^2$.
Therefore, given $\mathcal{A}$, the effect of $T_c$ on the right hand sum is at most
\[
\sum_{l=0}^{m-2}(\Delta^{1-l/2m}\hat{p})\Delta^{(l+1)/2m}\hat{p}^2
+ (\Delta^{1-(m-1)/2m}\hat{p})\Delta^{1/2}\hat{p}^2
= m\hat{p}^3\Delta^{1+1/2m}.
\]
Given $\mathcal{A}$, $T_c$ thus affects $D_1$ by at most 
\[
d_{G_c}(u)\hat{p}^2 + m\hat{p}^3\Delta^{1+1/2m}.
\]
Since $\sum_c d_{G_c}(u) \leq \Delta + \Delta_2 < 2\Delta$ and, by (P6), $d_{G_c}(u) \leq 3\omega_6 T\theta\Delta\hat{p}$,
\begin{align*}
&\sum_c (d_{G_c}(u)\hat{p}^2 + m\hat{p}^3\Delta^{1+1/2m})^2 \\
&\leq \hat{p}^4\sum_c d_{G_c}(u)^2 + 4m\hat{p}^5\Delta^{1+1/2m}\Delta + Cm^2\hat{p}^6\Delta^{2+1/m} \\
&\leq 3\hat{p}^5\omega_6T\theta\Delta\sum_c d_{G_c}(u) + 4m\hat{p}^5\Delta^{1+1/2m}\Delta + Cm^2\hat{p}^6\Delta^{2+1/m} \\
&\leq 6\omega_6T\theta\hat{p}^5\Delta^2 + 4m\hat{p}^5\Delta^{2+1/2m} + Cm^2\hat{p}^6\Delta^{2+1/m}.
\end{align*}
Together with Claim \ref{conccond} and (R10), Corollary \ref{bdbad} now implies
\begin{align*}
\Pr[D_1 > f_u(1-\theta/2) + 1/2\omega_3] 
&\leq \Pr[D_1 > f_u(1-9\theta/16)/\Pr[\mathcal{A}] + 1/2\omega_3] \\
&\leq \Pr[D_1 > \Ex[D_1]/\Pr[\mathcal{A}] + 1/2\omega_3] \\
&\comp{\leq}{C.\ref{bdbad}} e^{-1/4\omega_3^2 (6\omega_6T\theta\hat{p}^5\Delta^2 + 4m\hat{p}^5\Delta^{2+1/2m} + Cm^2\hat{p}^6\Delta^{2+1/m})} + \Pr[\bar{\mathcal{A}}] \\
&\comp{\leq}{(R10)} e^{-7\log\Delta} + \Pr[\bar{\mathcal{A}}] \\
&\comp{\leq}{C.\ref{conccond}} e^{-7\log\Delta} + e^{-10\log\Delta} \\
&< e^{-6\log\Delta}.
\end{align*}

We now bound $D_2$. 
We first prove that for any edge $uvw$,
\begin{equation}\label{exuvgivenw}
\Ex[p'_u(c)p'_v(c) | \gamma_w(c)=1] \leq p_u(c)p_v(c)(1+1/\omega_0).
\end{equation}
Assume that both $p'_u(c)$ and $p'_v(c)$ are determined by \eqref{flip1}.
If $c \in L(u)$ or $c \in L(v)$, then $p'_u(c)p'_v(c)=0$, 
so by \eqref{prcnotinuv},
\begin{align*}
\Ex[p'_u(c)p'_v(c) | \gamma_w(c)=1] 
&\leq \frac{p_u(c)}{q_u(c)}\frac{p_v(c)}{q_v(c)}\Pr[c \notin L(u) \cup L(v) | \gamma_w(c)=1] \\
&\leq \frac{p_u(c)}{q_u(c)}\frac{p_v(c)}{q_v(c)}\Pr[c \notin L(u) \cup L(v)] \\
&\comp{\leq}{\eqref{prcnotinuv}} p_u(c)p_v(c)(1+1/\omega_0).
\end{align*}
Suppose $p'_u(c)$ is determined by \eqref{flip1} 
and $p'_v(c)$ is determined by \eqref{flip2}.
Then $p'_u(c)$ and $p'_v(c)$ are independent of each other, and
$p'_v(c)$ is independent of the event $\gamma_w(c)=1$, so
\begin{align*}
\Ex[p'_u(c)p'_v(c) | \gamma_w(c)=1] 
&= \Ex[p'_u(c) | \gamma_w(c)=1]\Ex[p'_v(c)] \\
&\comp{=}{\eqref{expuc}} \Ex[p'_u(c) | \gamma_w(c)=1]p_v(c) \\
&\leq \frac{p_u(c)}{q_u(c)}\Pr[c \notin L(u) | \gamma_w(c)=1] p_v(c)\\
&\leq \frac{p_u(c)}{q_u(c)}\Pr[c \notin L(u)] p_v(c) \\
&= p_u(c)p_v(c) \\
&< p_u(c)p_v(c)(1+1/\omega_0).
\end{align*}
If $p'_u(c)$ and $p'_v(c)$ are both determined by \eqref{flip2}, then
\begin{align*}
\Ex[p'_u(c)p'_v(c) | \gamma_w(c)=1] 
= \Ex[p'_u(c)p'_v(c)]
= \Ex[p'_u(c)]\Ex[p'_v(c)]
\comp{=}{\eqref{expuc}} p_u(c)p_v(c),
\end{align*}
which establishes \eqref{exuvgivenw}.

Now, by \eqref{exuvgivenw},
\begin{align*}
\Ex[D_2] &= \sum_c \sum_{uvw} (\Ex[p'_u(c)p'_v(c)\I[\gamma_w(c)=1]] + 
\Ex[p'_u(c)p'_w(c)\I[\gamma_v(c)=1]]) \\
&= \sum_c \sum_{uvw} \Ex[p'_u(c)p'_v(c) | \gamma_w(c)=1]\Pr[\gamma_w(c)=1] \\
&+
\sum_c \sum_{uvw} \Ex[p'_u(c)p'_w(c) | \gamma_v(c)=1]\Pr[\gamma_v(c)=1] \\
&\leq (1+1/\omega_0)\sum_c \sum_{uvw} (p_u(c)p_v(c)\Pr[\gamma_w(c)=1] + p_u(c)p_w(c)\Pr[\gamma_v(c)=1]) \\
&= (1+1/\omega_0)\sum_c \sum_{uvw} (p_u(c)p_v(c)\theta p_w(c) + p_u(c)p_w(c)\theta p_v(c)) \\
&= (1+1/\omega_0)2\theta e_u.
\end{align*}

\noindent
Again, let 
\[
T_c = \{\gamma_v(c): v \in N(N(u))\} \cup \{\eta_v(c): v \in N(N(u))\}.
\]
Then $D_2$ is determined by the set of random variables $\{T_c: c \in C(u)\}$ .
Observe that
\[
D_2 = \sum_c \sum_{l=0}^{m-1} \sum_{v \in N_H(u) \cap N^0(u,l)}  \I[\gamma_v(c)=1] 
\sum_{w \in N_H(u,v)}p'_u(c)p'_w(c).
\]
The random variable $T_c$ does not affect terms of the form 
$\I[\gamma_v(d)=1]\sum_{w \in N(u,v)}p'_u(d)p'_w(d),$
where $d \neq c$.
$T_c$ affects the term $\I[\gamma_v(c)=1]\sum_{w \in N(u,v)}p'_u(c)p'_w(c)$ only if $\gamma_v(c)=1$;
in this case, the effect is at most $d_H(u,v)\hat{p}^2$.
Thus, given the event $\mathcal{A}$ from Claim \ref{conccond}, 
the total effect of $T_c$ on $D_2$ is bounded by
\begin{align*}
\sum_{l=0}^{m-2} \Delta^{1-l/2m}\hat{p} \Delta^{(l+1)/2m}\hat{p}^2 + \Delta^{1-(m-1)/2m}\hat{p}\delta \hat{p}^2
& < m \Delta^{1+1/2m}\hat{p}^3 + \delta\Delta^{1/2 + 1/2m} \hat{p}^3.
\end{align*}
By Corollary \ref{bdbad}, (R11), and Claim \ref{conccond},
\begin{align*}
\Pr[D_2 > 3\theta e_u + 1/2\omega_3] 
&\leq \Pr[D_2 > (1+1/\omega_0)2\theta e_u/\Pr[\mathcal{A}] + 1/2\omega_3] \\ 
&\comp{\leq}{C.\ref{bdbad}} e^{-2/(4\omega_3^2 C (m \Delta^{1+1/2m}\hat{p}^3 + \delta\Delta^{1/2 + 1/2m} \hat{p}^3)^2)} + \Pr[\bar{\mathcal{A}}]\\ 
&\comp{\leq}{(R11)} e^{-7\log\Delta} + \Pr[\bar{\mathcal{A}}] \\
&\comp{\leq}{C.\ref{conccond}} e^{-7\log\Delta} + e^{-10\log\Delta} \\
&\leq e^{-6\log\Delta}.
\end{align*}
Therefore, with probability at least $1-2\Delta^{-5}$,
\begin{align*}
f'_u 
&\leq f_u(1-\theta/2) + 1/2\omega_3 + 3\theta e_u + 1/2\omega_3 \\
&\leq f_u(1-\theta/2) + 3\theta e_u + 1/\omega_3.
\end{align*}

\noindent \textbf{Proof of (Q5)}.
Since
\[
d'_H(u) 
= \frac{1}{2}\sum_{v \in N_H(u)}\sum_{w \in N_H(u,v)} \I[v,w \in U']
\leq \frac{1}{2}\sum_{v \in N_H(u)}d_H(u,v)\I[v \in U'],
\]
Claim \ref{prcol} implies
\[
\Ex[d'_H(u)] \leq \frac{1-3\theta/4}{2}\sum_{v \in N_H(u)}d_H(u,v) = (1-3\theta/4)d_H(u).
\]
We prove concentration in the same way as in the proof of (Q3).
Let 
\[
T_c = \{\gamma_v(c): v \in N(N(u))\} \cup \{\eta_v(c): v \in N(N(u))\}.
\]
The random variable $d'_H(u)$ is determined by the
set of random variables $\{T_c: c \in C(u)\}$.
For $v \in N(u)$,  $T_c$ affects the term $d_H(u,v)\I[v \in V']$ only if
$\gamma_v(c)=1$, and in this case, the effect is at most $d_H(u,v)$. 
Thus, given the event $\mathcal{A}$ from Claim \ref{conccond}, 
$T_c$ affects $d'_H(u)$ by at most
\[
\sum_{l=0}^{m-2} \Delta^{1-l/2m}\hat{p}\Delta^{(l+1)/2m} + \Delta^{1-(m-1)/2m}\hat{p}\delta
< m \Delta^{1+1/2m}\hat{p} + \Delta^{1/2 + 1/2m} \hat{p}\delta.
\]
By Corollary \ref{bdbad}, (R14), and Claim \ref{conccond}, 
\begin{align*}
\Pr[d'_H(u) > (1-\theta/2)d_H(u)+\omega_5]
&\leq \Pr[d'_H(u) > (1-3\theta/4)d_H(u)/\Pr[\mathcal{A}]+\omega_5] \\
&\comp{\leq}{C.\ref{bdbad}} e^{-2\omega_5^2 / C(m \Delta^{1+1/2m}\hat{p} + \Delta^{1/2 + 1/2m} \hat{p}\delta)^2} + \Pr[\bar{\mathcal{A}}] \\
&\comp{\leq}{(R14)} e^{-7\log\Delta} + \Pr[\bar{\mathcal{A}}] \\
&\comp{\leq}{C.\ref{conccond}} e^{-7\log\Delta} + e^{-10\log\Delta} \\
&\leq e^{-6\log\Delta}.
\end{align*}

\subsection{Final Step}\label{sec_finalstep}
After the iterative portion of the algorithm, 
some vertices will still be uncolored.
Assuming (R1)-(R21) and Lemmas \ref{mainlemma}, \ref{auxlemma}, and \ref{bbound} hold,
we color them using the Asymmetric Local Lemma as follows.
Suppose $u$ has not been colored.
By (P1), (R4), Lemma \ref{bbound}, and (R18),
\begin{align*}
\sum_{c \in C(u) - B^T(u)}p^T_u(c) 
= \sum_{c \in C(u)}p^T_u(c) - \sum_{c \in B^T(u)}p^T_u(c)
&\comp{\geq}{(P1)} 1-T/\omega_1 - |B^T(u)|\hat{p} \\
&\comp{\geq}{(R4)} 1-o(1) - |B^T(u)|\hat{p} \\
&\comp{\geq}{L.\ref{bbound}} 1-o(1)-\epsilon \\
&\comp{\geq}{(R18)} 1/2.
\end{align*}
For each $c \notin B^T(u)$, define
\[
p_u^*(c) := \frac{p^T_u(c)}{\sum_{c \in C(u) - B^T(u)}p^T_u(c)} \leq 2p^T_u(c).
\]
For each uncolored vertex $u$, 
randomly assign $u$ one color from the distribution given by $p^*_u$.
For an edge $e = uvw \in H^T$, 
let $A_{uvw}$ denote the event that $u$, $v$, and $w$ receive the same color.
By (R7) and definition of $\theta$, $T/\omega_2 = o(\omega/\epsilon)$;
in particular, $T/\omega_2 = o(\omega)$.
So by (Q2), 
\[
e^T_{uvw} \leq e_{uvw}^0 + T/\Delta \omega_2 = 1/C^2 + o(\omega/\Delta) = \omega/\Delta + o(\omega/\Delta).
\]
Therefore
\[
\Pr[A_{uvw}] = \sum_c p^*_u(c)p^*_v(c)p^*_w(c) \leq 8\sum_c p^T_u(c)p^T_v(c)p^T_w(c) = 8e^T_{uvw} \leq 9\omega/\Delta.
\]
For each $c$ and each pair $uv \in G^T_c$, let
$B_{uv,c}$ denote the event that $u$ and $v$ both receive color $c$.
By (P3), for each $u$,
\[
\sum_{c \in C(u)} \sum_{ux \in G^T_c} \Pr[B_{ux,c}] 
\leq 4\sum_{c \in C(u)} \sum_{ux \in G^T_c} p^T_u(c)p^T_x(c)
= 4f^T_u
\leq 32(1-\theta/4)^{T} \omega.
\]
The event $A_{uvw}$ depends on any event $A_e$ or $B_{f,d}$, 
where $u$, $v$, or $w$ is in the edge $e$ or the edge $f$.
Using (P5),
\begin{align*}
&\sum_{e \in H^T: u \in e}\Pr[A_e] + \sum_{e \in H^T: v \in e}\Pr[A_e] 
 + \sum_{e \in H^T: w \in e}\Pr[A_e] \\
&+ \sum_{c \in C(u)} \sum_{ux \in G^T_c}\Pr[B_{ux,c}] + \sum_{c \in C(v)} \sum_{vx \in G^T_c}\Pr[B_{vx,c}]
 \sum_{c \in C(w)} \sum_{wx \in G^T_c}\Pr[B_{wx,c}] \\
&\leq 3(9\omega/\Delta)(1-\theta/3)^{T}\Delta + 3(32)(1-\theta/4)^{T} \omega \\ 
&\leq 123(1-\theta/4)^{T} \omega \\
&\leq 123 e^{-\theta T / 4} \omega \\
&= 123 e^{-5\log\omega / 4} \omega \\
&= 123 (\frac{1}{\omega})^{5/4} \omega \\
&< 1/4.
\end{align*}

The event $B_{uv, c}$ depends on any event $A_e$ or $B_{f,d}$, where $u$ or $v$ is in $e$ or $f$.
Since
\begin{align*}
&\sum_{e \in H^T: u \in e}\Pr[A_e] + \sum_{e \in H^T: v \in e}\Pr[A_e]
+ \sum_{c \in C(u)} \sum_{ux \in G^T_c}\Pr[B_{ux,c}] + \sum_{c \in C(v)} \sum_{vx \in G^T_c}\Pr[B_{vx,c}] \\
&\leq 18(1-\theta/3)^{T} \omega + 64(1-\theta/4)^{T} \omega \\ 
&\leq 1/4,
\end{align*}
the Asymmetric Local Lemma implies that there exists a coloring where none of the events 
$A_{uvw}$ or $B_{uv,c}$ occur.
Since no color in $B^T(u)$ and no color with $p^T_c(u) = 0$ was assigned to $u$,
this coloring, combined with the partial coloring from the algorithm, is a
proper list coloring of $H \cup G$.

\section {Triangle-free hypergraphs}\label{sec_trianglefree}
We will derive Theorem \ref{mainthm} as a corollary of the following theorem:
\begin{thm}\label{mainthm0}
Set $c_0 = 1/86,000$.
Suppose $H$ is a rank 3, triangle-free hypergraph with maximum $3$-degree at most $\Delta$,
maximum 2-degree at most $(c_0 \Delta \log\Delta)^{1/2}$, and maximum codegree at most $\Delta^{6/10}$. 
Then
\[
\chi_l(H) \leq (\frac{\Delta}{c_0 \log\Delta})^{1/2}.
\]
\end{thm}
\noindent
To prove this using Theorem \ref{analysisthm}, we need to find values
for the parameters $\omega$, $\epsilon$, $\omega_0$, and $\hat{p}$ which
satisfy (R1)-(R21), \eqref{prcnotinuvw}, \eqref{prcnotinuvh}, and \eqref{prcnotinuv}, 
and $\omega = c_0\log\Delta$.
We will show that the following values satisfy these criteria:
\begin{align*}
\epsilon &= 1/40 &  \omega &= (1/25)(\epsilon/86)\log\Delta & 
\hat{p} &= \Delta^{-11/24} & \omega_0 &= 1/19\theta\hat{p}.
\end{align*}

\noindent
By Claim \ref{paramsclaim}, these parameters satisfy (R1)-(R21), so
all that remains is to show that inequalities \eqref{prcnotinuvw}, 
\eqref{prcnotinuvh}, and \eqref{prcnotinuv} hold.
Fix a color $c$. 
In Claim \ref{trifreeclaim},  we first show that that hypergraph 
$H \cup G_c$ remains triangle-free throughout the algorithm.
The next three claims then show that if the hypergraph remains triangle-free,
we will have enough independence to derive 
\eqref{prcnotinuvw}, \eqref{prcnotinuvh}, and \eqref{prcnotinuv}.
Throughout the rest of this section, we will be taking intersections and unions over edges;
when we do this, we use the notation $e$ in place of $e \in E(H) \cup E(G_c)$.
\begin{claim}\label{trifreeclaim}
For iteration $i$, if $H^i \cup G^i_c$ is triangle-free, 
then $H^{i+1} \cup G^{i+1}_c$ is triangle-free.
\end{claim}
\begin{proof}
It suffices to show that when the algorithm creates $G^{i+1}_c$ from $G^i_c$ by
adding an edge $uv$ to $G^i_c$, no triangle is created.
Toward a contradiction, suppose that a triangle is created with
distinct edges $uv, e, f \in H^{i+1} \cup G^{i+1}_c$ and distinct vertices $u, v, w$ 
such that $u \in e$, $v \in f$, $w \in e \cap f$, and $u \notin f$, $v \notin e$.
Note that $u,v,w \in V(H^i \cup G^i_c)$ and $e,f \in H^i \cup G^i_c$.
Since $w \in V(H^i \cup G^i_c)$, $w$ has not been colored. Thus there exists a 
vertex $x \in V(H^i)-w$ 
and an edge $uvx \in H^i$ which gave rise to the edge $uv$. The edges
$uvx$, $e$, and $f$ form a triangle with vertices $u$, $v$, and $w$ in $H^i+G^i_c$, 
a contradiction.
\end{proof}
\noindent
In the rest of this section, we define
\[
d(u,v) = | \{e \in H \cup G_c: u,v \in e\} |.
\]
In addition, we drop the superscript from $H^i$ and $G_c^i$.

\begin{claim}\label{codegreeclaim}
Suppose $uvw \in H$,  $d(u, v) \geq 2$, and $d(w, v) \geq 2$. Then $d(u,w) = 1$.
\end{claim}
\begin{proof}
Since $d(u,v) \geq 2$ and $d(w,v) \geq 2$, there exist distinct edges $e,f \neq uvw$
such that $u,v \in e$ and $w,v \in f$.
If there exists $x \neq v$ such that $uwx \in H$, then $e, f$, and $uxw$ form a triangle
with corresponding vertices $u$, $v$, and $w$.
If $uw \in G_c$, then $e, f$, and $uw$ form a triangle with vertices $u$, $v$, and $w$.
\end{proof}

\begin{claim}
If $uvw$ is an edge and $d(u,w) = 1$, then
\begin{equation}\label{induw}
(\bigcup_{\substack{e: u \in e; v \notin e}}e-u) 
 \cap (\bigcup_{\substack{e: w \in e; v \notin e}}e-w) = \emptyset,
\end{equation}
\begin{equation}\label{induv}
(\bigcup_{\substack{e: u \in e; v \notin e}}e-u) 
 \cap (\bigcup_{e: v \in e; u \notin e}e-v) = \emptyset,
\end{equation}
and
\begin{equation}\label{indwv}
(\bigcup_{\substack{e: w \in e; v \notin e}}e-u) 
 \cap (\bigcup_{e: v \in e; w \notin e}e-v) = \emptyset.
\end{equation}
\end{claim}
\begin{proof}
Let $x \in U$, and let $e$ be an edge such that $u \in e$, $v \notin e$, and $x \in e-u$.
Then $e \neq uvw$, and since $d(u,w) = 1$, $x \notin \{u,v,w\}$.

Suppose $f$ is an edge such that $w \in f$, $v \notin f$, and $x \in f-w$.
Then, since $x \in f$, $f \neq uvw$.
Using $d(u,w) = 1$, $u \in e$, $w \in f$ and $e,f \neq uvw$,
we get $e \neq f$, $u \notin f$, and $w \notin e$.
Since $x \notin uvw$, we obtain a triangle with edges $e$, $f$, and $uvw$ 
and vertices $u$, $w$, and $x$.

Now suppose that $v, x \in f$ and $u \notin f$.
Again, $f \neq uvw$.
Because $u \in e$ and $u \notin f$, $e \neq f$.
Since $u \notin f$, $v \notin e$, and $x \notin \{u,v,w\}$,
$e$, $f$, and $uvw$ form a triangle with vertices $u$, $v$, and $x$.
By symmetry, this also gives \eqref{indwv}.
\end{proof}

\begin{claim}
If $uv \in G_c$, then
\begin{equation}\label{induvedge}
(\bigcup_{\substack{e: u \in e; v \notin e}}e-u) 
 \cap (\bigcup_{e: v \in e; u \notin e}e-v) = \emptyset.
\end{equation}
\end{claim}
\begin{proof}
If there exist edges $e$ and $f$ and a vertex $x$ such that $u \in e$, $v \notin e$, $v \in f$,
$u \notin f$, and $x \in e-u \cap f-v$, then $e$, $f$, and $uv$ form a triangle with
vertices $u$, $v$, and $x$ in $H \cup G_c$.
\end{proof}

\noindent For a set of vertices $S$, 
let $\gamma_S(c)=1$ denote the event that $\gamma_v(c)=1$ for all $v \in S$,
and let $\gamma_S(c) \neq 1$ denote the event that $\gamma_v(c)=0$ for some $v \in S$.
\begin{claim}\label{qxbound}
For any three vertices $x$, $y$, and $z$,
\[
\Pr[\bigcap_{\substack{e: x \in e; y \notin e}} \gamma_{e-x}(c) \neq 1] \leq
\Pr[\bigcap_{\substack{e: x \in e; y,z \notin e}} \gamma_{e-x}(c) \neq 1]
\leq q_x(c)(1+3\theta\hat{p}).
\]
\end{claim}
\begin{proof}
Note first that
\[
\Pr[\bigcap_{e: x \in e; y \in e} \gamma_{e-x}(c) \neq 1] \geq \Pr[\gamma_y(c) = 0] \geq 1-\theta\hat{p}.
\]
Similarly,
\[
\Pr[\bigcap_{e: x \in e; z \in e} \gamma_{e-x}(c) \geq 1-\theta\hat{p}.
\]
Since the events
$\bigcap_{x \in e; y \notin e} \gamma_{e-x}(c) \neq 1$ and
$\bigcap_{x \in e; y \in e} \gamma_{e-x}(c) \neq 1$
are monotone decreasing, the FKG inequality and then the previous two inequalities yield
\begin{align*}
q_x(c) 
&= \Pr[\bigcap_{\substack{e:x \in e; y,z \notin e}} \gamma_{e-x}(c) \neq 1 
 \bigcap_{e: x, y \in e} \gamma_{e-x}(c) \neq 1
 \bigcap_{e: x, z \in e} \gamma_{e-x}(c) \neq 1] \\
&\geq \Pr[\bigcap_{\substack{e: x \in e; y,z \notin e}} \gamma_{e-x}(c) \neq 1]
 \Pr[\bigcap_{e:x \in e, y \in e} \gamma_{e-x}(c) \neq 1]
 \Pr[\bigcap_{e:x \in e, z \in e} \gamma_{e-x}(c) \neq 1] \\
&\geq \Pr[\bigcap_{\substack{e:x \in e; y,z \notin e}} \gamma_{e-x}(c) \neq 1](1-\theta\hat{p})^2 \\
&\geq \Pr[\bigcap_{\substack{e:x \in e; y,z \notin e}} \gamma_{e-x}(c) \neq 1](1-2\theta\hat{p}).
\end{align*}
Thus
\[
\Pr[\bigcap_{\substack{e: x \in e; y,z \notin e}} \gamma_{e-x}(c) \neq 1]
\leq q_x(c)/(1-2\theta\hat{p})
\leq q_x(c)(1+3\theta\hat{p}).
\]
\end{proof}
We can now prove \eqref{prcnotinuvw}, \eqref{prcnotinuvh}, and \eqref{prcnotinuv}. 
Suppose $uvw$ is an edge.
By Claim \ref{codegreeclaim}, we may assume $d(u,w) = 1$.
The events
$\bigcap_{u \in e; v \notin e}\gamma_{e-u}(c) \neq 1$,
$\bigcap_{w \in e; v \notin e}\gamma_{e-w}(c) \neq 1$, and
$\bigcap_{v \in e; u,w \notin e}\gamma_{e-v}(c) \neq 1$
depend only on the sets of random variables
\[
\{\gamma_x(c): x \in \bigcup_{e: u\in e; v \notin e}e-u\},
\]
\[
\{\gamma_x(c): x \in \bigcup_{e: w\in e; v \notin e}e-w\},
\]
and
\[
\{\gamma_x(c): x \in \bigcup_{e: v\in e; u,w \notin e}e-v\},
\]
respectively.
By \eqref{induw}, \eqref{induv}, and \eqref{indwv}, these sets are pairwise disjoint,
so the three events are independent of each other. Therefore, applying Claim \ref{qxbound},
\begin{align*}
&\Pr[c \notin L(u) \cup L(v) \cup L(w)] \\
&= \Pr[\bigcap_{e: u \in e}\gamma_{e-u}(c) \neq 1 \bigcap_{e: v \in e}\gamma_{e-v}(c) \neq 1 \bigcap_{e: w \in e}\gamma_{e-w}(c) \neq 1] \\
&\leq \Pr[\bigcap_{\substack{e:u \in e; v \notin e}}\gamma_{e-u}(c) \neq 1 \bigcap_{\substack{e:v \in e; u,w \notin e}}\gamma_{e-v}(c) \neq 1 \bigcap_{\substack{e:w \in e; v \notin e}}\gamma_{e-w}(c) \neq 1] \\
&= \Pr[\bigcap_{\substack{e:u \in e; v \notin e}}\gamma_{e-u}(c) \neq 1]\Pr[\bigcap_{\substack{e:v \in e; u,w \notin e}}\gamma_{e-v}(c) \neq 1] \Pr[\bigcap_{\substack{e:w \in e; v \notin e}}\gamma_{e-w}(c) \neq 1] \\
&\comp{\leq}{C.\ref{qxbound}} q_u(c)q_v(c)q_w(c)(1+3\theta\hat{p})^3 \\
&< q_u(c)q_v(c)q_w(c)(1+19\theta\hat{p}) \\
&= q_u(c)q_v(c)q_w(c)(1+1/\omega_0).
\end{align*}
This proves \eqref{prcnotinuvw}. The proof of \eqref{prcnotinuvh} is the same,
except we start with any two vertices in $uvw$ instead of all three.

Suppose now that $uv \in G_c$ for some color $c$.
By \eqref{induvedge} and Claim \ref{qxbound},
\begin{align*}
\Pr[c \notin L(u) \cup L(v)] 
&= \Pr[\bigcap_{e: u \in e}\gamma_{e-u}(c) \neq 1 \bigcap_{e: v \in e}\gamma_{e-v}(c) \neq 1] \\
&\leq \Pr[\bigcap_{\substack{e:u \in e; v \notin e}}\gamma_{e-u}(c) \neq 1 \bigcap_{\substack{e:v \in e;u \notin e}}\gamma_{e-v}(c) \neq 1] \\
&\comp{=}{\eqref{induvedge}} \Pr[\bigcap_{e:\substack{u \in e; v \notin e}}\gamma_{e-u}(c) \neq 1]\Pr[\bigcap_{\substack{e:v \in e;u \notin e}}\gamma_{e-v}(c) \neq 1]\\
&\comp{\leq}{C.\ref{qxbound}} q_u(c)q_v(c)(1+3\theta\hat{p})^2 \\
&< q_u(c)q_v(c)(1+7\theta\hat{p}) \\
&< q_u(c)q_v(c)(1+1/\omega_0),
\end{align*}
completing the proof of \eqref{prcnotinuv} and Theorem \ref{mainthm0}.

\noindent \textbf{Proof of Theorem \ref{mainthm}}:
Recall that $c_0 = 1/86,000$.
Let $H$ be a rank $3$, triangle-free hypergraph with maximum $3$-degree $\Delta$ and 
maximum $2$-degree $\Delta_2$.
The original hypergraph $H$ may have some pairs of vertices with codegree
too large to apply Theorem \ref{mainthm0}, so we will work on a modified
hypergraph instead. 
Let
\[
K(u) = \{v \in N(u): d(u,v) \geq \Delta^{6/10}\}.
\]
Define a new hypergraph $H'$ with $V(H') = V(H)$ and
\[
E(H') = E(H) 
 - (\bigcup_{u \in V(H)} \bigcup_{v \in K(u)} \{e : u, v \in e\})
 + (\bigcup_{u \in V(H)} \bigcup_{v \in K(u)} \{u, v\})
\]
Let $\Delta'$, $\Delta_2'$, and $\delta'$ denote the maximum $3$-degree,
maximum $2$-degree, and maximum codegree of $H'$, respectively.
Note that $H'$ is still triangle-free, 
$\chi_l(H) \leq \chi_l(H')$, $\delta' \leq \Delta^{6/10}$, and $\Delta' \leq \Delta$.

Suppose $\Delta'_2 \leq \sqrt{\Delta}\sqrt{c_0 \log \Delta}$.
Since $\Delta' \leq \Delta$ and $\delta' \leq \Delta^{6/10}$,
Theorem \ref{mainthm0} implies
\[
\chi_l(H) \leq \chi_l(H') \leq (\frac{\Delta}{c_0 \log \Delta})^{1/2}.
\]

On the other hand, suppose $\Delta'_2 > \sqrt{\Delta}\sqrt{c_0 \log \Delta}$.
Then, since
\[
\Delta \geq d_H(u) \geq \frac{1}{2}\sum_{v \in N_H(u)}d_H(u,v) 
\geq \frac{1}{2}\sum_{\substack{v \in N_H(u)\\d_H(u,v) \geq \Delta^{6/10}}}d_H(u,v)
\geq |K(u)|\Delta^{6/10}/2,
\]
we have
\[
\Delta'_2 \leq \Delta_2 + 2\Delta^{4/10} < \Delta_2 + \Delta'_2/2.
\]
Choose $\Delta''$ so that $\Delta'_2 = \sqrt{\Delta''}\sqrt{c_0\log\Delta''}$.
Since $\Delta'_2 > \sqrt{\Delta}\sqrt{c_0\log \Delta}$, $\Delta'' > \Delta$.
Then the maximum $3$-degree of $H'$ is at most $\Delta < \Delta''$,
the maximum $2$-degree of $H'$ is at most 
$\Delta'_2 \leq \sqrt{\Delta''}\sqrt{c_0 \log \Delta''}$, and the maximum codegree
of $H'$ is at most $\Delta^{6/10} < \Delta''^{6/10}$,
so Theorem \ref{mainthm0} implies
\[
\chi_l(H) \leq \chi_l(H') \leq (\frac{\Delta''}{c_0 \log\Delta''})^{1/2} 
= \frac{\Delta'_2}{c_0 \log\Delta''} 
< \frac{\Delta'_2}{c_0 \log \Delta'_2}
< \frac{2\Delta_2}{c_0 \log 2\Delta_2}.
\]

\bibliographystyle{amsplain}
\bibliography{bib}
\end{document}